\documentclass[11pt,american]{amsart}

\usepackage{amssymb,babel,color}
\usepackage[T1]{fontenc}
\usepackage{microtype}
\definecolor{strongblue}{RGB}{0,0,160}
\definecolor{strongred}{RGB}{160,0,0}
\renewcommand*{\eqref}[1]{\textcolor{strongblue}{(\ref{#1})}}
\usepackage[ocgcolorlinks=true,citecolor=strongblue,linkcolor=strongblue,urlcolor=strongred,backref=page]{hyperref}
\numberwithin{equation}{section}
\usepackage[capitalize,nameinlink]{cleveref}
\usepackage[marginratio=1:1,paperwidth=460pt,paperheight=700pt,height=590pt,width=360pt]{geometry}
\usepackage[numeric,abbrev,nobysame,msc-links]{amsrefs}

\def\frak{\mathfrak}
\def\Bbb{\mathbb}
\def\Cal{\mathcal}

\let\phi\varphi

\newcommand{\x}{\times}
\renewcommand{\o}{\circ}

\newcommand{\ocirc}{\circledcirc}
\newcommand{\al}{\alpha}
\newcommand{\be}{\beta}
\newcommand{\ga}{\gamma}

\newcommand{\ep}{\epsilon}
\newcommand{\ka}{\kappa}
\newcommand{\la}{\lambda}
\newcommand{\om}{\omega}
\newcommand{\ph}{\phi}

\renewcommand{\th}{\theta}
\newcommand{\si}{\sigma}
\newcommand{\ze}{\zeta}
\newcommand{\Ga}{\Gamma}
\newcommand{\La}{\Lambda}
\newcommand{\Ph}{\Phi}

\newcommand{\Om}{\Omega}

\newcommand{\Ups}{\Upsilon}
\def\Rho{\mbox{\textsf{P}}}

\newcommand{\id}{\operatorname{id}}

\newcommand{\Ad}{\operatorname{Ad}}

\newcommand{\End}{\operatorname{End}}

\newcommand{\gr}{\operatorname{gr}}

\renewcommand*{\MR}[1]{\href{http://www.ams.org/mathscinet-getitem?mr=#1}{ MR~#1}}

\newcounter{theorem}
\numberwithin{theorem}{section}
\newtheorem{thm}[theorem]{Theorem}
\newtheorem*{thm*}{Theorem}
\newtheorem{lemma}[theorem]{Lemma}
\newtheorem{prop}[theorem]{Proposition}
\newtheorem{cor}[theorem]{Corollary}
\newtheorem*{lemma*}{Lemma \thesubsection}
\newtheorem*{prop*}{Proposition \thesubsection}
\newtheorem*{cor*}{Corollary}

\theoremstyle{definition}
\newtheorem{definition}[theorem]{Definition}
\newtheorem*{definition*}{Definition \thesubsection}
\newtheorem{example}[theorem]{Example}
\newtheorem*{example*}{Example \thesubsection}
\theoremstyle{remark}
\newtheorem{remark}[theorem]{Remark}
\newtheorem*{remark*}{Remark \thesubsection}

\makeatletter

\renewcommand*{\@seccntformat}[1]{%
   \csname the#1\endcsname.\hspace{1mm}} 
\def\@settitle{\begin{center}%
  \baselineskip14\p@\relax
    \sc
 \Large\@title
 \end{center}%
}
\def\@setauthors{%
  \begingroup
  \def\thanks{\protect\thanks@warning}%
  \trivlist
  \centering\footnotesize \@topsep30\p@\relax
  \advance\@topsep by -\baselineskip
  \item\relax
  \author@andify\authors
  \def\\{\protect\linebreak}%
  \sc\large{\authors}%
  \ifx\@empty\contribs
  \else
    ,\penalty-3 \space \@setcontribs
    \@closetoccontribs
  \fi
  \endtrivlist
  \endgroup
}
\def\@settitle{\begin{center}%
  \baselineskip14\p@\relax
    \sc
 \Large\@title
 \end{center}%
}
\def\@setauthors{%
  \begingroup
  \def\thanks{\protect\thanks@warning}%
  \trivlist
  \centering\footnotesize \@topsep30\p@\relax
  \advance\@topsep by -\baselineskip
  \item\relax
  \author@andify\authors
  \def\\{\protect\linebreak}%
  \sc\large{\authors}%
  \ifx\@empty\contribs
  \else
    ,\penalty-3 \space \@setcontribs
    \@closetoccontribs
  \fi
  \endtrivlist
  \endgroup
}

\makeatother

\begin{document}

\title{Geometric Theory of Weyl Structures
} 
\date{April 28, 2022}
\author{Andreas \v Cap}
\address{A.\v C.: Faculty of Mathematics\\
University of Vienna\\
Oskar--Morgenstern--Platz 1\\
1090 Wien\\
Austria}
\email{Andreas.Cap@univie.ac.at}
\author{Thomas Mettler}
\address{T.M.: Faculty of Mathematics, UniDistance Suisse, Schinerstrasse 18, 3900 Brig, Switzerland}
\email{thomas.mettler@fernuni.ch}

\begin{abstract}

Given a parabolic geometry on a smooth manifold $M$, we study a natural affine bundle $A \to M$, whose smooth sections can be identified with Weyl structures for the geometry. We show that the initial parabolic geometry defines a reductive Cartan geometry on $A$, which induces an almost bi-Lagrangian structure on $A$ and a compatible linear connection on $TA$. We prove that the split-signature metric given by the almost bi-Lagrangian structure is Einstein with non-zero scalar curvature, provided the parabolic geometry is torsion-free and $|1|$-graded. We proceed to study Weyl structures via the submanifold geometry of the image of the corresponding section in $A$. For Weyl structures satisfying appropriate non-degeneracy conditions, we derive a universal formula for the second fundamental form of this image. We also show that for locally flat projective structures, this has close relations to solutions of a projectively invariant Monge-Ampere equation and thus to properly convex projective structures.
\end{abstract}


\maketitle

\pagestyle{myheadings} \markboth{\v Cap and Mettler}{Weyl structures}
 
\section{Introduction}\label{1}

Parabolic geometries form a class of geometric structures that look very diverse in
their standard description. This class contains important and well-studied examples
like conformal and projective structures, non-degenerate CR structures of
hypersurface type, path geometries, quaternionic contact structures, and various
types of generic distributions. They admit a uniform conceptual description as Cartan
geometries of type $(G,P)$ for a semisimple Lie group $G$ and a parabolic subgroup
$P\subset G$ in the sense of representation theory. Such a geometry on a smooth manifold $M$ is given by a
principal $P$-bundle $p : \mathcal{G} \to M$ together with a~\textit{Cartan
  connection} $\omega \in \Omega^1(\mathcal{G},\mathfrak{g})$, which defines an
equivariant trivialization of the tangent bundle $T\Cal G$. A standard reference for
parabolic geometries is \cite{book}.

The group $P$ can be naturally written as a semi-direct product $G_0\ltimes P_+$ of a
reductive subgroup $G_0$ and a nilpotent normal subgroup $P_+$. For a Cartan geometry
$(p:\Cal G\to M,\om)$ the quotient $\Cal G_0:=\Cal G/P_+\to M$ is a principal
$G_0$-bundle, and some parts of $\om$ can be descended to that bundle. In the
simplest cases, this defines a usual first order $G_0$-structure on $M$, in more
general situations a filtered analog of such a structure. Thus the Cartan geometry
can be viewed as an extension of a first order structure. This reflects the fact that
morphisms of parabolic geometries are in general not determined locally around a
point by their 1-jet in that point, and the Cartan connection captures the necessary
higher order information.

To work explicitly with parabolic geometries, one often chooses a more restrictive
structure, say a metric in a conformal class, a connection in a projective class or a
pseudo-Hermitian structure on a CR manifold, expresses things in terms of this choice
and studies the effect of different choices. It turns out that there is a uniform way
to do this that can be applied to all parabolic geometries, namely the concept of Weyl
structures introduced in \cite{Weyl}, see Chapter 5 of \cite{book} for an improved
exposition. Choosing a Weyl structure, one in particular obtains a linear connection
on any natural vector bundle associated to a parabolic geometry, as well as an
identification of higher order geometric objects like tractor bundles with more
traditional natural bundles. The set of Weyl structures always forms an affine space
modeled on the space on one-forms on the underlying manifold, and there are explicit
formulae for how a change of Weyl structure affects the various derived quantities.

The initial motivation for this article were the results in \cite{Dunajski-Mettler}
on projective structures. Such a structure on a smooth manifold $M$ is given by an
equivalence class $[\nabla]$ of torsion-free connections on its tangent bundle, where
two connections are called equivalent if they have the same geodesics up to
parametrization. While these admit an equivalent description as a parabolic geometry,
the underlying structure $\Cal G_0\to M$ is the full frame bundle of $M$ and thus
contains no information. Hence there is the natural question, whether a projective
structure can be encoded into a first order structure on some larger space
constructed from $M$. Indeed, in \cite{Dunajski-Mettler}, the authors associate to a
projective structure on an $n$-dimensional manifold $M$ a certain rank $n$ affine bundle $A \to M$, whose total space can be canonically endowed with a
neutral signature metric $h$, as well as a non-degenerate $2$-form $\Omega$. It turns
out that the metric $h$ is Einstein and $\Omega$ is closed. Moreover, the pair
$(h,\Omega)$ is related by an endomorphism of $TA$ which squares to the identity map
and its eigenbundles $L^{\pm}$ are Lagrangian with respect to $\Omega$. Equivalently,
we may think of the pair $(h,\Omega)$ as an \textit{almost para-K\"ahler structure} or as an \textit{almost bi-Lagrangian
  structure} $(\Omega,L^+,L^-)$ on $A$, see \cref{3.1} for the formal definition and more details. 

In addition, it is observed that the sections of $A \to M$ are in bijective
correspondence with the connections in the projective class. Consequently, all
the submanifold notions of symplectic -- and pseudo-Riemannian geometry can be
applied to the representative connections of $[\nabla]$. This leads in particular to
the notion of a minimal Lagrangian connection~\cite{Mettler:MinLag}. As detailed
below, this concept has close relations to the concept of properly convex
projective structures. These in turn provide a connection to the study of representation varieties and higher Teichm\"uller spaces, see~\cite{arXiv:1803.06870} for a survey. 

\medskip

In an attempt to generalize these constructions to a larger class of parabolic
geometries, we were led to a definition of $A\to M$ that directly leads to an
interpretation as a bundle of Weyl structures. This means that the space of sections
of $A\to M$ can be naturally identified with the space of Weyl structures for the
geometry $(p:\Cal G\to M,\om)$. At some stage it was brought to our attention that a
bundle of Weyl structures had been defined in that way already in the article
\cite{Herzlich} by M.~Herzlich in the setting of general parabolic geometries. In
this article, Herzlich gave a rather intricate argument for the existence of a
connection on $TA$ and used this to study canonical curves in parabolic geometries.

The crucial starting point for our results here is that a parabolic geometry $(p:\Cal
G\to M,\om)$ can also naturally be interpreted as a Cartan geometry on $A$ with
structure group $G_0$. This immediately implies that for any type of parabolic
geometry, there is a canonical linear connection on any natural vector bundle over
$A$ as well as natural almost bi-Lagrangian structure on $A$ that is compatible with
the canonical connection. So in particular, we always obtain a non-degenerate two-form
$\Om\in\Om^2(A)$, a neutral signature metric $h$ on $TA$ as well as a decomposition
$TA=L^-\oplus L^+$ as a sum of Lagrangian subbundles.

Using the interpretation via Cartan geometries, it turns out that all elements of the
theory of Weyl structures admit a natural geometric interpretation in terms of
pulling back operations on $A$ via the section defined by a Weyl structure. This
works for general parabolic geometries as shown in \cref{2}. In
particular, we show that Weyl connections are obtained by pulling back the canonical
connection on $A$, while the Rho tensor (or generalized Schouten tensor) associated to
a Weyl connection is given by the pullback of a canonical $L^+$-valued one-form on
$A$.

We believe that this interpretation of Weyl structures should be a very useful
addition to the tool set available for the study of parabolic geometries. Indeed,
working with the canonical geometric structures on $A$ compares to the standard way of using Weyl structures, like working on a frame bundle compares to working in local
frames.

For the second part of the article, we adopt a different point of
view. From \cref{3} on, we use the relation to Weyl structures as a tool for the
study of the intrinsic geometric structure on $A$ and its relation to non-linear
invariant PDE. Our first main result shows that one has to substantially restrict the
class of geometries in order to avoid getting into exotic territory. Recall that for
a parabolic subgroup $P\subset G$ the corresponding Lie subalgebra $\frak
p\subset\frak g$ can be realized as the non-negative part in a grading $\frak
g=\oplus_{i=-k}^k\frak g_i$ of $\frak g$, which is usually called a
$|k|$-grading. There is a subclass of parabolic geometries that is often referred to
as \textit{AHS structures}, see e.g.~\cite{BastonI,BastonII,AHS1}, which is the case
$k=1$, see \cref{2.2} and \cref{rem3.1} for more details. This is exactly the case in
which the underlying structure $\Cal G_0\to M$ is an ordinary first order
$G_0$-structure. In particular, there is the notion of intrinsic torsion for this
underlying structure. Vanishing of the intrinsic torsion is equivalent to the
existence of a torsion free connection compatible with the structure and turns out to
be equivalent to torsion-freeness of the Cartan geometry $(p:\Cal G\to M,\om)$. Using
this background, we can formulate the first main result of \cref{3}, that we prove as
\cref{thm3.1}:

\begin{thm*}
    Let $(p:\Cal G\to M,\om)$ be a parabolic geometry of type $(G,P)$ and $\pi:A\to M$ its
  associated bundle of Weyl structures. Then the natural 2-form $\Om\in\Om^2(A)$ is closed if and only if $(G,P)$ corresponds to a $|1|$-grading
  and the Cartan geometry $(p:\Cal G\to M,\om)$ is torsion-free.
\end{thm*}

Hence we restrict our considerations to torsion-free AHS structures from this point
on. Apart from projective and conformal structures, this contains also
Grassmannian structures of type $(2,n)$ and quaternionic structures, for which there
are many non-flat examples. For several other AHS structures, torsion-freeness
implies local flatness, but the locally flat case is of particular interest for us
anyway. Our next main result, which we prove in \cref{thm3.3}, vastly generalizes
\cite{Dunajski-Mettler}:

\begin{thm*}
  For any torsion-free AHS structure, the pseudo-Riemannian metric $h$ induced by the
  canonical almost bi-Lagrangian structure on the bundle $A$ of Weyl structures is an
  Einstein metric with non-zero scalar curvature. 
\end{thm*}

While one could prove the aforementioned Theorems on a case by case basis by using the techniques from \cite{Dunajski-Mettler}, our arguments instead rely on a careful analysis of the properties of the curvature tensor of the induced connection on $TA$. Following \cite{Mettler:MinLag}, we next initiate the study of Weyl structures via
the geometry of submanifolds in $A$. We call a Weyl structure $s : M \to A$ of a
torsion-free AHS structure~\textit{Lagrangian} if $s : M \to (A,\Omega)$ is a
Lagrangian submanifold. Likewise, $s$ is called~\textit{non-degenerate} if $s : M \to
(A,h)$ is a non-degenerate submanifold. We show that a Weyl structure is Lagrangian
if and only if its Rho tensor is symmetric and that it is non-degenerate if and only
if the symmetric part of its Rho tensor is non-degenerate. 

In \cref{thm3.4.1} we characterize Lagrangian Weyl structures that lead to totally
geodesic submanifolds $s(M)\subset A$, which provides a connection to Einstein metrics
and reductions of projective holonomy. If $s$ is non-degenerate, then
there is a well defined second fundamental form of $s(M)$ with respect to any linear
connection on $TA$ that is metric for $h$ and we show that this admits a natural
interpretation as a $\binom12$-tensor field on $M$. In our next main result,
\cref{thm3.4.2}, we give explicit formulae for the second fundamental forms of the canonical connection
and the Levi-Civita connection of $h$. These are universal formulae in terms of the
Weyl connection, the Rho-tensor, and its inverse, which are valid for all
torsion-free AHS structures. As an application, we are able to characterize
non-degenerate Lagrangian Weyl structures that are minimal submanifolds in $(A,h)$ in
terms of a universal PDE. Again, this is a vast generalization of
\cite[Theorem 4.4]{Mettler:MinLag}, where merely the case of projective structures
on surfaces was considered.

In \cref{4} we connect our results to the study of fully non-linear invariant PDE on
AHS structures. A motivating example arises from the work of E.~Calabi. In
\cite{MR0365607}, Calabi related complete affine hyperspheres to solutions of a
certain Monge-Amp\`ere equation. This Monge-Amp\`ere equation, when interpreted
correctly, is an invariant PDE that one can associate to a projective structure and
it is closely linked to properly convex projective manifolds,
see~\cite[Theorem 4]{MR1828223}. In \cref{thm4.3}, we relate Calabi's equation to our equation for a minimal
Lagrangian Weyl structure and as \cref{cor:propconvex}, we obtain:
\begin{cor*}
Let $(M,[\nabla])$ be a closed oriented locally flat projective manifold. Then $[\nabla]$ is
properly convex if and only if $[\nabla]$ arises from a minimal Lagrangian Weyl
structure whose Rho tensor is positive definite.
\end{cor*}
The convention for the Rho tensor used here is chosen to be consistent with \cite{book}. This convention is natural from a Lie theoretic viewpoint, but differs from the standard definition in projective -- and conformal differential geometry by a sign. It should also be noted that a relation between properly convex projective manifolds and
minimal Lagrangian submanifolds has been observed previously
in~\cite{MR2854277,MR2854275} (but not in the context of Weyl structures).

The notion of convexity for projective structures is only defined for locally flat structures. The above Corollary thus provides a way to generalize the notion of a properly convex projective structure to a class of projective structures that are possibly curved, namely those arising from a minimal Lagrangian Weyl structure. Going beyond projective geometry, this class of differential geometric structures is well-defined for all torsion-free AHS structures. 

We conclude the article by showing that there are analogs of the projective Monge-Amp\`ere equation for other AHS structures, and that these always can be described in terms of the Rho tensor, which provides a relation to submanifold geometry of Weyl structures. These topics will be studied in detail elsewhere.

\subsection*{Acknowledgments} A\v C gratefully acknowledges support by the Austrian
Science Fund (FWF): P 33559-N. A part of the research for this article was carried
out while TM was visiting Forschungsinstitut f\"ur Mathematik (FIM) at ETH
Z\"urich. TM thanks FIM for its hospitality and DFG for partial funding through the
priority programme SPP 2026 ``Geometry at Infinity''. We also thank M.~Dunajski,
J.\ \v Silhan and V.\ \v Zadn\`\i k for helpful discussions and the anonymous referee
for interesting comments that helped improving the article significantly.

\section{The bundle of Weyl structures}\label{2}

This section works in the setting of general parabolic geometries. We assume that the
reader is familiar with the basic concepts and only briefly collect what we need
about parabolic geometries and Weyl structures. Then we define the bundle of Weyl
structures and identify some of the geometric structures that are naturally induced
on its total space. We then prove existence of a canonical connection and explain how
these structures can be used as an equivalent encoding of the theory of Weyl
structures.

\subsection{Parabolic geometries}\label{2.1}
The basic ingredient needed to specify a type of parabolic geometry is a semisimple
Lie algebra $\frak g$ that is endowed with a so-called $|k|$-grading. This is a
decomposition
$$
\frak g=\frak g_{-k}\oplus\dots\oplus\frak g_{-1}\oplus\frak g_0\oplus\frak
g_1\oplus\dots\oplus\frak g_k
$$
of $\frak g$ into a direct sum of linear subspaces such that
\begin{itemize}
\item $[\frak g_i,\frak g_j]\subset\frak g_{i+j}$, where we agree that $\frak
  g_\ell=\{0\}$ for $|\ell|>k$.
\item No simple ideal of $\frak g$ is contained in the subalgebra $\frak g_0$.
\item The subalgebra $\frak p_+=\frak g_1\oplus\dots\oplus\frak g_k$ is generated by
  $\frak g_1$.  
\end{itemize}
In particular, this implies that the Lie subalgebra $\frak g_0$ naturally acts on
each of the spaces $\frak g_i$ via the restriction of the adjoint action. Moreover,
$\frak p:=\frak g_0\oplus\frak p_+$ is a Lie subalgebra of $\frak g$, which turns out
to be a parabolic subalgebra in the sense of representation theory.

Such $|k|$-gradings can be easily described in terms of the structure theory of
semisimple Lie algebras, see Section 3.2 of \cite{book}. In particular, it turns out
that any parabolic subalgebra is obtained in this way and, essentially, the
classification of gradings is equivalent to the classification of parabolic
subalgebras. Further, the decomposition $\frak p=\frak g_0\oplus\frak p_+$ is the
reductive Levi decomposition, so it is a semi-direct product, $\frak p_+$ is the
nilradical of $\frak p$, and the subalgebra $\frak g_0$ is reductive. Of
course, also $\frak g_-:=\frak g_{-k}\oplus\dots\oplus\frak g_{-1}$ is a Lie
subalgebra of $\frak g$, which is nilpotent by the grading property. It turns out
that $\frak g_-$ and $\frak p_+$ are isomorphic.

Next, one chooses a Lie group $G$ with Lie algebra $\frak g$. Then the normalizer of
$\frak p$ in $G$ has Lie algebra $\frak p$, and one chooses a closed subgroup
$P\subset G$ lying between this normalizer and its connected component of the
identity. The subgroup $P$ naturally acts on $\frak g$ and $\frak p$ via the adjoint
action. More generally, one puts $\frak g^i:=\oplus_{j\geq i}\frak g_j$ to define a
filtration of $\frak g$ by linear subspaces that is invariant under the adjoint
action of $P$. This makes $\frak g$ into a filtered Lie algebra in the sense that
$[\frak g^i,\frak g^j]\subset\frak g^{i+j}$.

Having made these choices, there is the concept of a \textit{parabolic geometry of
  type} $(G,P)$ on a manifold $M$ of dimension $\dim(G/P)$. This is defined as a
Cartan geometry $(p:\Cal G\to M,\om)$ of type $(G,P)$, which means that $p:\Cal G\to
M$ is a principal $P$-bundle and that $\om\in\Om^1(\Cal G,\frak g)$ is a
\textit{Cartan connection}. This in turn means that $\om$ is equivariant for the
principal right action, so $(r^g)^*\om=\Ad(g^{-1})\o\om$, reproduces the generators
of fundamental vector fields, and that $\om(u):T_u\Cal G\to\frak g$ is a linear
isomorphism for each $u\in\Cal G$. In addition, one requires two conditions on the
curvature of $\om$, which are called \textit{regularity} and \textit{normality},
which we don't describe in detail. 

While Cartan geometries provide a nice uniform description of parabolic geometries,
this should be viewed as the result of a theorem rather than a definition. To proceed
towards more common descriptions of the geometries, one first observes that the
Lie group $P$ can be decomposed as a semi-direct product. On the one hand, the
exponential map restricts to a diffeomorphism from $\frak p_+$ onto a closed normal
subgroup $P_+\subset P$. On the other hand, one defines a closed subgroup $G_0\subset
P$ as consisting of those elements, whose adjoint action preserves the grading of
$\frak g$ and observes that this has Lie algebra $\frak g_0$. Then the inclusion of
$G_0$ into $P$ induces an isomorphism $G_0\to P/P_+$.

Using this, one can pass from the Cartan geometry $(p:\Cal G\to M,\om)$ to an
underlying structure by first forming the quotient $\Cal G_0:=\Cal G/P_+$, which is a
principal $G_0$-bundle. Moreover, for each $i=-k,\dots,k$, there is a smooth
subbundle $T^i\Cal G\subset T\Cal G$ consisting of those tangent vectors that are
mapped to $\frak g^i\subset\frak g$ by $\om$. Since $T^1\Cal G$ is the vertical
bundle of $\Cal G\to\Cal G_0$, these subbundles descend to a filtration $\{T^i\Cal
G_0:i=-k,\dots,0\}$ of $T\Cal G_0$. Moreover, for each $i<0$, the component of $\om$
in $\frak g_i$ descends to define a smooth section of the bundle $L(T^i\Cal G_0,\frak
g_i)$ of linear maps, so this can be viewed as a partially defined $\frak g_i$-valued
differential form.

The simplest case here is $k=1$, for which the geometries in question are often
referred to as \textit{AHS structures}. In this case, one obtains a
$\frak g_{-1}$-valued one-form $\th$ on $\Cal G_0$, which is $G_0$-equivariant and
whose kernel in each point is the vertical subbundle. This means that
$(p_0:\Cal G_0\to M,\th)$ in this case simply is a classical first order structure
corresponding to the adjoint action of $G_0$ on $\frak g_{-1}$ (which turns out to be
infinitesimally effective). According to a result of Kobayashi and Nagano (see
\cite{KN1}), the resulting class of structures for simple $\frak g$ is very peculiar,
since these are the only irreducible first order structures of finite type, for which
the first prolongation is non-trivial. This class contains important examples, like
conformal structures, almost quaternionic structures, and almost Grassmannian
structures.

For general $k$, there is an interpretation of $\Cal G_0$ and the partially defined
forms as a filtered analogue of a first order structure. This involves a filtration of
the tangent bundle $TM$ by smooth subbundles $T^iM$ for $i=-k,\dots,-1$ with
prescribed integrability and non-integrability properties together with a reduction of structure
group of the associated graded vector bundle to the tangent bundle. This leads to
examples like hypersurface-type CR structures, in which the filtration is equivalent
to a contact structure, while the reduction of structure group is defined by an
almost complex structure on the contact subbundle. Further important example of such
structures are path geometries, quaternionic contact structures and various types of
generic distributions.

Except for two cases, the Cartan geometry can be uniquely (up to isomorphism)
recovered from the underlying structure (see Section 3.1 of \cite{book}), and indeed
this defines an equivalence of categories. So in this case, one has two equivalent
descriptions of the structure. The two exceptional cases are projective structures and a
contact analogue of those. In these cases, the underlying structure contains no
information respectively describes only the contact structure, and one in addition
has to choose an equivalence class of connections in order to describe the
structure. Still, these fit into the general picture with respect to Weyl structures,
which we discuss next.

\subsection{Weyl structures}\label{2.2}
These provide the basic tool to explicitly translate between the description of a
parabolic geometry as a Cartan geometry and the picture of the underlying
structure. So let us suppose that $(p:\Cal G\to M,\om)$ is a Cartan geometry of type
$(G,P)$ and that $p_0:\Cal G_0\to M$ is the underlying structure described in
\cref{2.1}. The original definition of a Weyl structure used in \cite{Weyl} is as a
$G_0$-equivariant section $\si$ of the natural projection $q:\Cal G\to \Cal G_0=\Cal
G/P_+$. One shows that such sections always exist globally and by definition, they
provide reductions of the principal $P$-bundle $p:\Cal G\to M$ to the structure group
$G_0\subset P$. As a representation of $G_0$, the Lie algebra $\frak g$ splits as
$\frak g_-\oplus\frak g_0\oplus\frak p_+$ (and indeed further according to the
$|k|$-grading). Thus, the pullback $\si^*\om$ splits accordingly into a sum of three
$G_0$-equivariant one-forms with values in $\frak g_-$, $\frak g_0$ and $\frak p_+$,
respectively, which then admit nice interpretations in terms of the underlying
structure. The $\frak g_0$-component defines a principal connection on $\Cal G_0$,
which induces the \textit{Weyl connections} on associated bundles. The component in
$\frak p_+$ descends to a one-form on $M$ with values in the
associated graded bundle to the
cotangent bundle $T^*M$, which is the \textit{Rho-tensor} associated to the Weyl
structure. The $\frak g_-$-component also descends to $M$ and provides an isomorphism
between the tangent bundle $TM$ and its associated graded bundle. For the structures
we consider in this article, this component coincides with the soldering form that
identifies $\Cal G_0$ as a reduction of structure group of $TM$.

As observed in \cite{Herzlich}, \textit{any} reduction of $p:\Cal G\to M$ to the
structure group $G_0\subset P$ comes from a Weyl structure. This is because the
composition of $q$ with the principal bundle morphism defining such a reduction
clearly is an isomorphism of $G_0$-principal bundles. Thus one could equivalently
define a Weyl structure as such a reduction of structure group and then observe that
this defines a $G_0$-equivariant section of $q:\Cal G\to\Cal G_0$. It is a classical
result that reductions of $\Cal G$ to the structure group $G_0$ can be equivalently
described as smooth sections of the associated bundle with fiber $P/G_0$. This
motivates the following definition from \cite{Herzlich}.

\begin{definition}\label{def2.2} 
  The \textit{bundle of Weyl structures} associated to the parabolic geometry
  $(p:\Cal G\to M,\om)$ is $\pi:A:=\Cal G\x_P(P/G_0)\to M$.
\end{definition}

The correspondence between Weyl structures and smooth sections of $\pi:A\to M$ can be
easily made explicit. Given a $G_0$-equivariant section $\si:\Cal G_0\to\Cal G$ one
considers the map sending $u_0\in \Cal G_0$ to the class of $(\si(u_0),eG_0)$ in
$\Cal G\x_P(P/G_0)$, where $e\in P$ is the neutral element. By construction, the
resulting smooth map $\Cal G_0\to A$ is constant on the fibers of $p_0:\Cal G_0\to M$
and thus descends to a smooth map $s:M\to A$, which is a section of $\pi$ by
construction. Conversely, a section $s$ of $\pi$ corresponds to a smooth,
$G_0$-equivariant map $f:\Cal G\to P/G_0$ characterized by the fact that $s(x)$ is
the class of $(u,f(u))$ for each $u$ in the fiber of $\Cal G$ over $x$. But then
$f^{-1}(eG_0)$ is a smooth submanifold of $\Cal G$ on which the projection
$q:\Cal G\to\Cal G_0$ restricts to a $G_0$-equivariant diffeomorphism. The inverse of
this diffeomorphism gives the Weyl structure determined by $s$.

From the definition, we can verify that the bundle of Weyl structures is similar to
an affine bundle. This will also provide the well known affine structure on Weyl
structures in our picture. To formulate this, recall first that the parabolic
subgroup $P\subset G$ is a semi-direct product of the subgroup $G_0\subset P$ and the
normal subgroup $P_+\subset P$. In particular, any element $g\in P$ can be uniquely
written as $g_0g_1$ with $g_0\in G_0$ and $g_1\in P_+$, compare with Theorem 3.1.3 of
\cite{book}, and of course $g_0g_1=(g_0g_1g_0^{-1})g_0$ provides the corresponding
decomposition in the opposite order.

\begin{prop}\label{prop2.2}
Let $\pi:A\to M$ be the bundle of Weyl structures associated to a parabolic geometry
$(p:\Cal G\to M,\om)$. Then sections of $\pi:A\to M$ can be naturally identified with
smooth functions $f:\Cal G\to P_+$ such that $f(u\cdot
(g_0g_1))=g_1^{-1}g_0^{-1}f(u)g_0$ for each $u\in\Cal G$, $g_0\in G_0$ and $g_1\in
P_+$.

Fixing one function $f$ that satisfies this equivariancy condition, any other
function which is equivariant in the same way can be written as $\hat
f(u)=f(u)h(u)$, where $h:\Cal G\to P_+$ is a smooth function such that
$h(u\cdot(g_0g_1))=g_0^{-1}h(u)g_0$.
\end{prop}
\begin{proof}
The inclusion $P_+\hookrightarrow P$ induces a smooth map $P_+\to P/G_0$, and from
the decomposition of elements of $P$ described above, we readily see that this is
surjective. On the other hand, writing the quotient projection $P\to G_0$ as $\al$,
the map $g\mapsto g\al(g)^{-1}$ induces a smooth inverse, so $P/G_0$ is diffeomorphic
to $P_+$. Since $A=\Cal G\x_P(P/G_0)$, smooth sections of $\pi:A\to M$ are in
bijective correspondence with $P$-equivariant smooth functions $\Cal G\to P/G_0$, so
these can be viewed as functions with values in $P_+$. The equivariancy condition
reads as $f(u\cdot (g_0g_1))=g_1^{-1}g_0^{-1}\cdot f(u)$. But starting from $\tilde
g_1G_0$, we get $g_1^{-1}g_0^{-1}\tilde g_1G_0=g_1^{-1}g_0^{-1}\tilde g_1g_0G_0$, and
$g_1^{-1}g_0^{-1}\tilde g_1g_0\in P_+$. This completes the proof of the first claim.

Given one function $f:\Cal G\to P_+$, of course any other such function can be
uniquely written as $\hat f=fh$ for a smooth function $h:\Cal G\to P_+$, so it
remains to understand $P$-equivariance. What we assume is that
$f(u\cdot(g_0g_1))= g_1^{-1}g_0^{-1}f(u)g_0$ and we want $\hat f$ to satisfy the
analogous equivariancy condition. But this exactly requires that
$g_1^{-1}g_0^{-1}f(u)g_0h(u\cdot(g_0g_1))= g_1^{-1}g_0^{-1}f(u)h(u)g_0$, which is
equivalent to the claimed equivariancy of $h$.
\end{proof}

To connect to the well-known affine structure on the set of Weyl structures, we
observe two alternative ways to express things using the exponential map. On the one
hand, we have observed above that $\exp:\frak p_+\to P_+$ is a diffeomorphism. Thus
we can write $h(u)=\exp(\Ups(u))$ and equivariancy of $h$ is equivalent to
$\Ups(u\cdot(g_0g_1))=\Ad(g_0)^{-1}(\Ups(u))$. On the other hand, in the proof of
Theorem 3.1.3 of \cite{book} it is shown that also
$(Z_1,\dots,Z_k)\mapsto\exp(Z_1)\cdots\exp(Z_k)$ defines a diffeomorphism
$\frak g_1\oplus\dots\oplus \frak g_k\to P_+$. Correspondingly, we can write
$h(u)=\exp(\Ups_1(u))\cdots\exp(\Ups_k(u))$ where $\Ups_i:\Cal G\to\frak g_i$ is a
smooth map for each $i=1,\dots,k$. Again, equivariancy of $h$ translates to
$\Ups_i(u\cdot(g_0g_1))=\Ad(g_0)^{-1}(\Ups_i(u))$ for each $i$.

There is also a nice global way to express the affine structure. The filtration of
$TM$ induced by a parabolic geometry dualizes to a filtration of the cotangent bundle
$T^*M$ and we can form the associated graded bundle $\gr(T^*M)$. The general theory
implies that this can be realized as $\Cal G\x_{P}\gr(\frak p_+)\cong\Cal
G_0\x_{G_0}\frak p_+$.
\begin{prop}\label{prop2.2a}
Let $\pi:A\to M$ be the bundle of Weyl structures associated to a parabolic geometry
$(p:\Cal G\to M,\om)$. Then for any smooth section $s$ of $\pi$, there is an induced
diffeomorphism $\phi_s:T^*M\to A$.
\end{prop}
\begin{proof}
Let $\si:\Cal G_0\to\Cal G$ be the $G_0$-equivariant section determined by $s$. Since
$\exp:\frak p_+\to P_+$ is a diffeomorphism, we conclude that
$\Phi_s(u_0,Z):=\si(u_0)\exp(Z)$ defines a diffeomorphism $\Phi_s:\Cal G_0\x\frak
p_+\to \Cal G$. Given $g_0\in G_0$, the definition readily implies that $\Phi_s(u\cdot
g_0,\Ad(g_0^{-1})(Z))=\Phi_s(u_0,Z)\cdot g_0$. Hence there is an induced
diffeomorphisms between the orbit spaces $\gr(T^*M)=\Cal G_0\x_{G_0}\frak p_+$ and
$A=\Cal G_0/G_0$.
\end{proof}

\subsection{The basic geometric structures on \texorpdfstring{$A$}{A}}\label{2.3} 
It was shown in \cite{Herzlich} that the parabolic geometry $(p:\Cal G\to M,\om)$
gives rise to a connection on the tangent bundle $TA$ of $A$. The argument used to
obtain this connection is rather intricate: There is the opposite parabolic subgroup
$P^{op}$ to $P$ which corresponds to the Lie subalgebra
$\frak g_-\oplus\frak g_0\subset\frak g$ and one considers the homogeneous space
$G/P^{op}$. Restricting the $G$-action to $P$ and forming the associated bundle
$\Cal G\x_P(G/P^{op})$ the Cartan connection $\om$ induces a natural affine
connection on the total space of this bundle. It is then easy to see that
$P\cap P^{op}=G_0$, so acting with $P$ on $eP^{op}$ defines a $P$-equivariant open
embedding $A\to\Cal G\x_P(G/P^{op})$, thus providing a connection on $TA$ as
claimed. Our first main result provides a more conceptual description of this
connection, which directly implies compatibility with several additional geometric
structures on $A$.

\begin{prop}\label{prop2.3}
  The canonical projection $\Cal G\to A$ is a $G_0$-principal bundle and $\om$
  defines a Cartan connection on that bundle, so $(\Cal G\to A,\om)$ is a Cartan
  geometry of type $(G,G_0)$. In particular, the $\frak g_0$-component of $\om$
  defines a canonical principal connection on $\Cal G\to A$ and
  $TA\cong\Cal G\x_{G_0}(\frak g/\frak g_0)$, so this inherits a canonical linear
  connection. Finally, there is a natural splitting $TA=L^-\oplus L^+$ into a direct
  sum of two subbundles of rank $\dim(M)$, which is parallel for the connection and
  such that $L^+$ is the vertical bundle of $\pi$.
\end{prop}
\begin{proof}
  Mapping $u\in\Cal G$ to the class of $(u,eG_0)$ in $\Cal G\x_P(P/G_0)$ is
  immediately seen to be surjective and its fibers coincide with the orbits of $G_0$
  on $\Cal G$. Hence one obtains an identification of $\Cal G/G_0$ with $A$, and it
  is well known that this makes the projection $\Cal G\to A$ into a $G_0$-principal
  bundle. The defining properties of $\om$ for the group $P$ and the Lie algebra
  $\frak p$ then imply the corresponding properties for the group $G_0$ and the Lie
  algebra $\frak g_0$, so $\om$ defines a Cartan connection on $\Cal G\to A$. 

  As a representation of $G_0$, we get
  $\frak g=\frak g_0\oplus(\frak g_-\oplus\frak p_+)$. This means that we have given
  a $G_0$-invariant complement to $\frak g_0$ in $\frak g$. Decomposing $\om$
  accordingly, the component $\om_0$ in $\frak g_0$ is $G_0$-equivariant, thus
  defining a principal connection on $\Cal G\to A$, which induces linear connections
  on all associated vector bundles. Moreover, since $\om$ is a Cartan connection on
  $\Cal G\to A$, we can identify $TA$ with the associated vector bundle 
$$
\Cal G\x_{G_0}(\frak g/\frak g_0)\cong\Cal G\x_{G_0}(\frak g_-\oplus\frak p_+).
$$
This readily implies both the existence of a natural connection and of a compatible decomposition
of $TA$ with $L^-=\Cal G\x_{G_0}\frak g_-$ and $L^+=\Cal G\x_{G_0}\frak p_+$. The
tangent map $T\pi:TA\to TM$ is induced by the projection $\frak g/\frak g_0\to\frak
g/\frak p$. Identifying $\frak g/\frak g_0$ with $\frak g_-\oplus\frak p_+$ the
kernel of this projection is $\frak p_+$, which shows that $L^+\subset TA$ coincides
with $\ker(T\pi)$.
\end{proof}

This result also gives us a basic supply of natural vector bundles on $A$, namely the
vector bundles associated to the principal bundle $\Cal G\to A$ via representations
of $G_0$. Moreover, the principal connection on that bundle coming from $\om_0$ gives
rise to an induced linear connection on each of these associated bundles. We will
denote all these induced connections by $D$. Given an associated bundle $E\to A$, we
can view $D$ as an operator $D:\Ga(E)\to\Ga(T^*A\otimes E)$. Of course the splitting
$TA=L^-\oplus L^+$ from \cref{prop2.3} induces an analogous splitting of
$T^*A$, which allows us to split $D$ into two partial connections $D=D^-\oplus D^+$.
Here $D^\pm:\Ga(E)\to\Ga((L^\pm)^*\otimes E)$. Viewing $D$ as a covariant derivative,
$D^\pm$ is defined by differentiating only in directions of the corresponding
subbundle of $TA$.

\subsection{Relations between natural vector bundles}\label{2.4}

Recall that the natural vector bundles for the parabolic geometry
$(p:\Cal G\to M,\om)$ are the associated vector bundles of the form
$\Cal VM=\Cal G\x_P\Bbb V$ for representations $\Bbb V$ of $P$. Throughout this
article, we will only consider the case that the center $Z(G_0)$ of the subgroup
$G_0\subset P$ acts diagonalizably on $\Bbb V$. Together with the fact that $G_0$ is
reductive, this implies that $\Bbb V$ is completely reducible as a representation of
$G_0$. One important subclass of natural bundles is formed by \textit{completely
  reducible} bundles that correspond to representations of $P$ on which the subgroup
$P_+\subset P$ acts trivially, which is equivalent to complete reducibility as a
representation of $P$.  On the other hand, there are \textit{tractor bundles} which
correspond to restrictions to $P$ of representations of $G$.

Any representation $\Bbb V$ of $P$ can naturally be endowed with a $P$-invariant
filtration of the form $\Bbb V=\Bbb V^0\supset\Bbb V^1\supset\dots\supset\Bbb V^N$ as
follows (see Section 3.2.12 of \cite{book}). The smallest component $\Bbb V^N$
consists of those elements, on which $\frak p_+$ acts trivially under the
infinitesimal action. The larger components are characterized iteratively by the fact
that $v\in\Bbb V^j$ if and only if it is sent to $\Bbb V^{j+1}$ by the action of any
element of $\frak p_+$. Then one defines the \textit{associated graded}
representation $\gr(\Bbb V):=\oplus_{i=0}^N\gr_i(\Bbb V)$ with $\gr_i(\Bbb V):=(\Bbb
V^i/\Bbb V^{i+1})$ and $\Bbb V^{N+1}=\{0\}$. By construction, this is a completely
reducible representation of $P$.

As an important special case, consider the restriction of the adjoint representation
of $G$ to $P$. Then it turns out that, up to a shift in degree, the canonical
$P$-invariant filtration is exactly the filtration $\{\frak g^i\}$ derived from the
$|k|$-grading of $\frak g$ as in \cref{2.1}. In particular, this implies that
$\frak g^2=[\frak p_+,\frak p_+]$ and similarly, the higher filtrations components form
the lower central series of $\frak p_+$. Using this it is easy to see that the
natural filtration on any representation $\Bbb V$ of $P$ has the property that
$\frak g^i\cdot\Bbb V^j\subset\Bbb V^{i+j}$ for all $i,j\geq 0$ under the
infinitesimal representation of $\frak p=\frak g^0$. This readily implies that there
is a natural action of the associated graded $\gr(\frak p)$ on $\gr(\mathbb{V}
)$, which is
compatible with the grading. Since the filtration of $\frak p$ is induced by the
(non-negative part of the) grading on $\frak g$, we can identify $\gr(\frak p)$ with
$\frak p$ via the inclusion of $\frak g_i$ into $\frak g^i$. Altogether, we get, for
each $i,j\geq 0$, bilinear maps $\frak g_i\x\gr_j(\Bbb V)\to\gr_{i+j}(\Bbb V)$, which
are $P$-equivariant (with trivial action of $P_+$) by construction.

As a representation of the subgroup $G_0\subset P$, the associated graded $\gr(\mathbb V)$ is 
isomorphic to $\Bbb V$. Indeed, we have observed above that $\mathbb V$ is completely reducible 
as a representation of $G_0$, so the same holds for each of the subrepresentations 
$\mathbb V^j\subset\mathbb V$. In particular, there always is a $G_0$-invariant complement 
$\mathbb V_j$ to the invariant subspace $\mathbb V^{j+1}\subset\mathbb V^j$ and we put 
$\mathbb V_N=\mathbb V^N$. By construction, we on the one hand get $\mathbb V\cong\oplus\mathbb V_j$ 
and on the other hand $\mathbb V_j\cong\mathbb V^j/\mathbb V^{j+1}$ which implies the claimed 
statement. Otherwise put, one can interpret the passage
from $\Bbb V$ to $\gr(\Bbb V)$ as keeping the restriction to $G_0$ of the $P$-action
on $\Bbb V$ and extending this by the trivial action of $P_+$ to a new action of $P$.

The construction of the associated graded has a direct counterpart on the level of
associated bundles. Putting $\Cal VM:=\Cal G\x_P\Bbb V\to M$, any of the filtration components $\Bbb V^i$ defines a smooth subbundle $\Cal V^iM:=\Cal G\x_P\Bbb V^i\to
M$. Thus $\Cal VM$ is filtered by the smooth subbundles $\Cal V^iM$ and we can form
the associated graded vector bundle $\gr(\Cal VM)=\oplus (\Cal V^iM/\Cal
V^{i+1}M)$. It is easy to see that this can be identified with the associated bundle
$\Cal G\x_P\gr(\Bbb V)$. However, the fact that $\Bbb V$ and $\gr(\Bbb V)$ are
isomorphic as representations of $G_0$ does not have a geometric interpretation
without making additional choices. Hence on the level of associated bundles, it is
very important to carefully distinguish between a filtered vector bundle and its
associated graded.

Any representation $\Bbb V$ of $P$ defines a representation of $G_0$ by
restriction. Hence denoting by $\pi:A\to M$ the bundle of Weyl structures, $\Bbb V$
also gives rise to a natural vector bundle over $A$ that we denote by
$\Cal VA:=\Cal G\x_{G_0}\Bbb V\to A$. Some information is lost in that process,
however, for example $\Cal G\x_{G_0}\Bbb V\cong\Cal G\x_{G_0}\gr(\Bbb V)$ for any
representation $\Bbb V$ of $P$. Next, sections of $\Cal VA\to A$ can be naturally
identified with smooth functions $\Cal G\to\Bbb V$ that are
$G_0$-equivariant. Similarly, sections of $\Cal VM\to M$ are in bijective
correspondence with smooth functions $\Cal G\to\Bbb V$, which are
$P$-equivariant. Thus we see that there is a natural inclusion of $\Ga(\Cal VM\to M)$
as a linear subspace of $\Ga(\Cal VA\to A)$. We will denote this by putting a tilde
over the name of a section of $\Cal VM\to M$ in order to indicate the corresponding
section of $\Cal VA\to A$. So both the sections of $\Cal VM\to M$ and of its
associated graded vector bundle can be interpreted as (different) subspaces of the
space of sections of $\Cal VA\to A$.

Now we can describe the relations of bundles and sections explicitly.

\begin{thm}\label{thm2.4}
  Let $(p:\Cal G\to M,\om)$ be a parabolic geometry of type $(G,P)$ and let $\pi:A\to
  M$ be the corresponding bundle of Weyl structures. Fix a representation $\Bbb V$ of
  $P$ and consider the corresponding natural bundles $\Cal VM=\Cal G\x_P\Bbb V\to M$
  and $\Cal VA=\Cal G\x_{G_0}\Bbb V\to A$. Then we have:

  (1) $\Cal VA$ can be naturally identified with the pullback bundle $\pi^*\Cal
  VM$. In particular, $L^-\cong\pi^*TM$ and $L^+\cong\pi^*T^*M$.

  (2) The operation $\si\mapsto\tilde\si$ identifies $\Ga(\Cal VM\to M)$ with the
  subspace of $\Ga(\Cal VA\to A)$ consisting of those sections $\tau$ for which
  $D^+_{\ph}\tau=-\phi\bullet\tau$ for all $\ph\in\Ga(L^+)$. Here $\bullet:L^+\x\Cal
  VA\to\Cal VA$ is induced by the infinitesimal representation $\frak p_+\x\Bbb
  V\to\Bbb V$. 

  In particular, for a completely reducible bundle $\Cal VM$,
  $\tilde\si=\pi^*\si$ and the image consists of all sections that are parallel for
  $D^+$.

  (3) Any section $s:M\to A$ of $\pi$ determines a natural pullback operator
  $s^*:\Ga(\Cal VA\to A)\to\Ga(\gr(\Cal VM)\to M)$. In particular, choosing $s$,
  $\si\mapsto s^*\tilde\si$ defines a map $\Ga(\Cal VM)\to\Ga(\gr(\Cal VM))$. This
  map is induced by a vector bundle isomorphism $\Cal VM\to\gr(\Cal VM)$ that
  coincides with the isomorphism determined by the Weyl structure corresponding to
  $s$ as in Section 5.1.3 of \cite{book}.
\end{thm}
\begin{proof}
  (1) follows directly from the construction: Mapping a $G_0$-orbit in
  $\Cal G\x\Bbb V$ to the $P$-orbit it generates, defines a bundle map
  $\Cal VA\to\Cal VM$ with base map $\pi:A\to M$. This evidently restricts to a
  linear isomorphism in each fiber and hence defines an isomorphism
  $\Cal VA\to\pi^*\Cal VM$. The second statement follows from the well known facts that
  $TM\cong\Cal G\x_P(\frak g/\frak p)$ and $T^*M\cong\Cal G\x_P\frak p_+$ and the
  fact that $\frak g/\frak p\cong\frak g_-$ as a representation of $G_0$.

(2) Since $P$ is a semi-direct product, $P$-equivariancy of a function is equivalent
  to equivariancy under $G_0$ and $P_+$ and equivariancy under $P_+$ is equivalent to
  equivariancy for the infinitesimal action of $\frak p_+$. Hence for a
  $G_0$-equivariant function $f:\Cal G\to\Bbb V$, $P$-equivariancy is equivalent to
  the fact that for each $u\in\Cal G$ and $Z\in\frak p_+$ with fundamental vector
  field $\ze_Z$, we get $\ze_Z(u)\cdot f=-Z\cdot f(u)$. Here in the left hand side the
  vector field differentiates the function, while in the right hand side we use the
  infinitesimal representation of $\frak p_+$ on $\Bbb V$. Suppose that $u$ projects to $y\in A$. Then by definition, $\zeta_Z(u)$ is the horizontal lift with 
respect to $D$ of a tangent vector $\phi\in L^+_y\subset T_yA$. Hence $\zeta_Z(u)\cdot f$ represents 
$D_\phi\tau(y)=D^+_\phi\tau(y)$ in the, while $Z\cdot f(u)$ of course represents $\phi\bullet\tau(y)$.

  In the case of a completely reducible bundle, $\bullet$ is the zero map, so we see
  that our subspace coincides with the $D^+$-parallel sections. On the other hand,
  for any section $\si\in\Ga(\Cal VM)$, the pullback $\pi^*\si$ is constant along the
  fibers of $\pi$. Since we know from \cref{prop2.3} that $L^+$ is the
  vertical subbundle of $\pi$, this implies that $\pi^*\si=\tilde\si$.

(3) As we have noted already, for any representation $\Bbb V$ of $P$, the associated
  graded $\gr(\Bbb V)$ is isomorphic to $\Bbb V$ as a representation of $G_0$. Thus
  we conclude from (1) that we can not only identify $\Cal VA$ with the pullback of
  $\Cal VM$ but also with the pullback of the associated graded vector bundle
  $\gr(\Cal VM)$. Hence for a smooth section $s:M\to A$ and a point $x\in M$, we can
  naturally identify the fiber $\Cal V_{s(x)}A$ with the fiber over $x$ of $\gr(\Cal
  VM)$. This provides a pullback operator $s^*:\Ga(\Cal VA)\to\Ga(\gr(\Cal VM))$, so
  $\si\mapsto s^*\tilde\si$ defines an operator $\Ga(\Cal VM)\to\Ga(\gr(\Cal
  VM))$. This operator is evidently linear over $C^\infty(M,\Bbb R)$ and thus induced
  by a vector bundle homomorphism $\Cal VM\to\gr(\Cal VM)$ with base map
  $\id_M$. Suppose that for $\si\in\Ga(\Cal VM)$ and $x\in M$ we have
  $(s^*\tilde\si)(x)=0$. Then the function $f:\Cal G\to\Bbb V$ which corresponds to
  both $\si$ and $\tilde\si$ has to vanish along the fiber of $\Cal G\to A$ over
  $s(x)$. But $P$-equivariancy then implies that $f$ vanishes along the fiber of
  $\Cal G\to M$ over $x$, so $\si(x)=0$. This implies that our bundle map is
  injective in each fiber and since both bundles have the same rank it is an
  isomorphism of vector bundles.

  The standard description of the isomorphism $\Cal VM\to\gr(\Cal VM)$ induced by a
  Weyl structure is actually also phrased in the language of sections; Given the
  $P$-equivariant function $f:\Cal G\to\Bbb V$ corresponding to $\si$ and the
  $G_0$-equivariant section $\overline{s}:\Cal G_0\to\Cal G$ determined by $s$, one
  considers the $G_0$-equivariant function $f\o\overline{s}$. This describes a
  section of $\Cal G_0\x_{G_0}\Bbb V\cong\Cal G_0\x_{G_0}\gr(\Bbb V)$. Going through
  the identifications, it is clear that this coincides with the isomorphism described
  above.   
\end{proof}

\begin{remark}\label{rem2.4}
  (1) In principle, the pullback operation defined in part (3) of
  \cref{thm2.4} could also be interpreted as having values in $\Ga(\Cal VM\to M)$.
  Since $\Cal VA$ does not contain any information about the $\frak p_+$-action on
  $\Bbb V$, the interpretation with values in $\Ga(\gr(\Cal VM)\to M)$ seems much more
  natural to us.

  (2) The comparison to the standard description of Weyl structures in part (3) of
  the theorem also implies how the isomorphisms $\Cal VM\to\gr(\Cal VM)$ induced by
  sections $s$ of $A\to M$ are compatible with the affine structure on the space of
  these sections from \cref{prop2.2}, compare with Proposition 5.1.5 of
  \cite{book}. It is also easy to give a direct proof of this result in our
  picture. One just has to interpret the affine structure in terms of sections of
  $L^+\to A$ and then use the obvious solution of the differential equation
  $D^+_\ph\tau=-\ph\bullet\tau$ for appropriate sections $\ph$.
\end{remark}

\subsection{The Weyl connections}\label{2.5}
We next describe the interpretation of Weyl connections in our picture. At the same
time, we obtain a nice description of the Rho-corrected derivative associated to a
Weyl structure, that was first introduced in \cite{CDS}, see Section 5.1.9 of
\cite{book} for a discussion. The Rho-corrected derivative comes from a principal connection on $\Cal G$ determined by an equivariant section $\sigma:\Cal G_0\to\Cal G$. One takes the component of $\omega$ in $\frak p$ along the image of $\sigma$ and extends it equivariantly to a principal connection. The name "Rho-corrected derivative" comes from the explicit formula of this derivative in terms of Weyl connection an the Rho-tensor. To obtain our description, we first observe that the pullback
operation from part (3) of \cref{thm2.4} clearly extends to differential forms
with values in a natural vector bundle. Let $\Bbb V$ be a representation of $P$ with
corresponding natural bundles $\Cal VM\to M$ and $\Cal VA\to A$. Then one can pull
back a $\Cal VA$-valued $k$ form $\ph$ on $A$ along a section $s:M\to A$ to a
$\gr(\Cal VM)$-valued $k$-form $s^*\ph$ on $M$ in an obvious way.

\begin{thm}\label{thm2.5}
Let $\Bbb V$ be a representation of $P$ and let $\Cal VM\to M$ and $\Cal VA\to A$ be
the corresponding natural bundles. For $\si\in\Ga(\Cal VM)$ consider the natural lift
$\tilde\si\in\Ga(\Cal VA)$. For a smooth section $s:M\to A$ let $\nabla^s$ be the
Weyl connection of the Weyl structure determined by $s$. Let $\xi\in\frak X(M)$ be a
vector field with natural lift $\tilde\xi\in\Ga(L^-)$.

(1) The pullback $s^*D\tilde\si\in\Om^1(M,\gr(\Cal VM))$ of
$D\tilde\si\in\Om^1(A,\Cal VA)$ coincides with the image of
$\nabla^s\si\in\Om^1(M,\Cal VM)$ under the isomorphism $\Cal VM\to\gr(\Cal VM)$
induced by $s$ as in \cref{thm2.4}.

(2) The pullback $s^*(D^-_{\tilde\xi}\tilde\si)\in\Ga(\gr(\Cal VM))$ coincides with
the image of the Rho-corrected derivative $\nabla^{\Rho}_\xi\si\in\Ga(\Cal VM)$ under
the isomorphism induced by $s$ as in \cref{thm2.4}.
\end{thm}
\begin{proof}
  Let $\bar s:\Cal G_0\to\Cal G$ be the $G_0$-equivariant section corresponding to
  $s$. For a point $x\in M$, an element $u\in\Cal G$ with $p(u)=x$ lies in the image
  of $\bar s$ if and only if $u$ projects to $s(x)\in A$. Assuming this, put
  $u_0=q(u)$ where $q:\Cal G\to\Cal G_0$ is the projection, so $u=\bar s(u_0)$. To
  compute $\nabla^s$, we need the horizontal lift $\hat\xi\in\frak X(\Cal G_0)$ of
  $\xi$ for the principal connection $\bar s^*\om_0$. This is characterized by the fact that $\hat\xi(u_0)$ projects onto $\xi(x)$ and that
  $\om(u)(T_{u_0}\bar s\cdot\hat\xi(u_0))$ has vanishing $\frak g_0$-component. But
  by construction $T_{u_0}\bar s\cdot\hat\xi(u_0)$ projects onto $T_xs\cdot\xi(x)$
  and so vanishing of the $\frak g_0$-component implies that this is the horizontal
  lift of $T_xs\cdot\xi(x)$ in $u$ corresponding to the principal connection $\om_0$
  that induces $D$. From this, (1) follows immediately.

  The argument for (2) is closely similar. By definition, $\tilde\xi(s(x))$ is the
  unique tangent vector that lies in $L^-$ and projects onto $\xi(x)$. The
  $D$-horizontal lift of this tangent vector in $u$, by construction, is mapped to
  $\frak g_-$ by $\om$ and projects onto $\xi(x)\in T_xM$. But this is exactly the
  characterizing property of the horizontal lift with respect to the principal
  connection $\ga^{\bar s}$ used in Section 5.1.9 of \cite{book} to define the
  Rho-corrected derivative. Thus the restriction of the $G_0$-equivariant function
  $\Cal G\to\Bbb V$ representing $D^-_{\tilde\xi}\tilde\si$ to $\bar s(\Cal G_0)$
  coincides with the restriction of the $P$-equivariant function representing
  $\nabla^{\Rho}_\xi\si$ and the claim follows from \cref{thm2.4}.
\end{proof}

\subsection{The universal Rho-tensor}\label{2.6}
Using the pullback of bundle valued forms, we can also describe the Rho tensor in our
picture. Recall that we use the convention of \cite{Weyl} and \cite{book} for
Rho tensors in the setting of general parabolic geometries, which differs by sign from
the standard conventions for projective and conformal structures. 

\begin{prop}\label{prop2.6}
Let us view the projection $TA\to L^+$ as $\Rho\in\Om^1(A,L^+)$. Then for each smooth
section $s:M\to A$, the pullback $s^*\Rho\in\Om^1(M,\gr(T^*M))$ coincides with the
Rho-tensor of the Weyl structure determined by $s$ as defined in Section 5.1.2 of
\cite{book}.
\end{prop}
\begin{proof}
  Take a point $x\in M$, a tangent vector $\xi\in T_xM$ and consider
  $T_xs\cdot\xi\in T_{s(x)}A$. Choose a point $u\in\Cal G$ over $s(x)$ and consider
  its image $u_0=q(u)\in\Cal G_0$. Since $u$ projects to $s(x)$, it lies in the image
  of the $G_0$-equivariant section $\bar s:\Cal G_0\to\Cal G$ determined by $s$, so
  $u=\bar s(u_0)$. Taking a tangent vector $\hat\xi\in T_{u_0}\Cal G_0$, the
  tangent vector $T_{u_0}\bar s\cdot\hat\xi\in T_u\Cal G$, by construction, projects
  onto $T_xs\cdot\xi\in T_{s(x)}A$. But then, by definition, the $L^+$ component of
  $T_xs\cdot\xi$ is obtained by projecting
  $\om(u)^{-1}(\om_+(T_{u_0}\bar s\cdot\hat\xi))$ to $T_{s(x)}A$, where $\om_+$
  denotes the $\frak p_+$-component of the Cartan connection $\om$. But the
  Rho-tensor of $\bar s$ is defined as the $\gr(T^*M)$-valued form induced by the
  $G_0$-equivariant form $\bar s^*\om_+$, which completes the argument.
\end{proof}

\begin{definition}\label{def2.5}
  The form $\Rho\in\Om^1(A,L^+)$ defined by the projection $TA\to L^+$ is called the
  \textit{universal Rho-tensor} of the parabolic geometry $(p:\Cal G\to M,\om)$.
\end{definition}

\subsection{Curvature and torsion quantities}\label{2.7}
The curvature $K\in\Om^2(\Cal G,\frak g)$ of the Cartan connection $\om$ is defined
by $K(\xi,\eta)=d\om(\xi,\eta)+[\om(\xi),\om(\eta)]$ for $\xi,\eta\in\frak X(\Cal
G)$. Since $K$ is horizontal and $P$-equivariant, it can be interpreted as
$\ka\in\Om^2(M,\Cal AM)$, where $\Cal AM=\Cal G\x_P\frak g$ is the adjoint tractor
bundle. In the same way, we can interpret it as a two-form on $A$ with values in the
associated bundle $\Cal G\x_{G_0}\frak g$. Since $K$ is horizontal over $M$, it
follows that this two form vanishes upon insertion of one tangent vector from
$L^+\subset TA$. In view of the $G_0$-invariant decomposition $\frak g=\frak
g_-\oplus\frak g_0\oplus\frak p_+$ we can decompose that two-form further. To do
this, we denote by $\End_0(TA)$ the associated bundle $\Cal G\x_{G_0}\frak g_0$. Via
the adjoint action, this can naturally be viewed as a subbundle of $L(TA,TA)$.

\begin{definition}\label{def2.7}
  The components $T\in \Om^2(A,L^-)$, $W\in\Om^2(A,\End_0(TA))$ and
  $Y\in \Om^2(A,L^+)$ of the two form on $A$ induced by $K$ are called the
  \textit{universal torsion}, the \textit{universal Weyl curvature} and the
  \textit{universal Cotton-York tensor} of the parabolic geometry
  $(p:\Cal G\to M,\om)$.
\end{definition}

The following result follows directly from the definitions. 

\begin{prop}\label{prop2.7}
  For any smooth section $s:M\to A$, the pullbacks $s^*T\in\Om^2(M,\gr(TM))$, 
  $s^*W\in\Om^2(M,\End_0(TM))$ and $s^*Y\in\Om^2(M,\gr(T^*M))$ correspond to the
  components of the Cartan curvature $\ka\in\Om^2(M,\Cal AM)$ under the isomorphism
  $\Cal AM\cong\gr(\Cal AM)\cong \gr(TM)\oplus\End_0(TM)\oplus\gr(T^*M)$ induced by
  the Weyl structure determined by $s$. 
\end{prop}

These quantities are related to data associated to the Weyl structure determined by
$s$ in Section 5.2.9 of \cite{book} and these results can be easily recovered in the
current context.  

\medskip

On the level of $A$, the best way to interpret the components of the Cartan curvature
is via the torsion and curvature of the canonical connection $D$. This interpretation
will also be crucial for the analysis of the intrinsic geometric structure on $A$ in
\cref{3} below. To formulate the result, we need a bit more notation. The Lie bracket
is a $G$-equivariant, skew symmetric bilinear map $\frak g\x\frak g\to\frak g$. Now
we can restrict this to entries from $\frak g_-\oplus\frak p_+$ and then decompose the
values according to $\frak g=(\frak g_-\oplus\frak p_+)\oplus\frak g_0$, and the
result will still be $G_0$-equivariant. The first component induces a two-form on $A$
with values in $TA$ which we denote by $\{\ ,\ \}$. Similarly, the $\frak
g_0$-component of the bracket defines a two-form $\{\ ,\ \}_0$ on $A$ with values in
$\End_0(TA)$. Using this, we formulate

\begin{thm}\label{thm2.7}
  Let $A\to M$ be the bundle of Weyl structures associated to a parabolic geometry,
  and let $D$ be the canonical connection on $TA$. Let $\tau\in\Om^2(A,TA)$ be the
  torsion and $\rho\in\Om^2(A,L(TA,TA))$ be the curvature of $D$. Then we have:

  (1) The $TA$-valued two form $\tau+\{\ ,\ \}$ vanishes upon insertion of one
  section of $L^+$. On $\La^2L^-$, its components in $L^-$ and $L^+$ are the tensors
  $T$ and $Y$ from \cref{def2.7}, respectively.

  (2) The curvature $\rho$ has values in $\End_0(TA)\subset L(TA,TA)$. Moreover,
  $\rho+\{\ ,\ \}_0$ vanishes upon insertion of one section of $L^+$ and coincides
  with the tensor $W$ from \cref{def2.7} on $\La^2L^-$.
\end{thm}
\begin{proof}
  This follows from well known results on the curvature and torsion of the affine
  connection induced by a reductive Cartan geometry. For $\xi\in\frak X(A)$, let
  $\xi^h\in\frak X(\Cal G)$ be the horizontal lift. Then $\om(\xi^h):\Cal G\to (\frak
  g_-\oplus\frak p_+)$ is the equivariant function corresponding to $\xi$. Taking a
  second field $\eta$, the bracket $[\xi^h,\eta^h]$ lifts $[\xi,\eta]$ so
  $\om_{\pm}([\xi^h,\eta^h])$ is the equivariant function representing $[\xi,\eta]$.

  (1) From these considerations and the definition of the exterior derivative,  it
  follows readily that $d\om_{\pm}(\xi^h,\eta^h)$ is the equivariant function
  representing $\tau(\xi,\eta)$. On the other hand, the component of
  $[\om(\xi^h),\om(\eta^h)]$ in $\frak g_-\oplus\frak p_+$ of course represents
  $\{\xi,\eta\}$, so the claim follows from the definition of the curvature of a
  Cartan connection.

  (2) It is also well known that $-\om_0([\xi^h,\eta^h])$ is the function representing
  $\rho(\xi,\eta)$. Since the $\frak g_0$-component of $[\om(\xi^h),\om(\eta^h)]$
  clearly represents $\{\xi,\eta\}_0$, the result again follows from the definition
  of the Cartan curvature. 
\end{proof}

\section{The natural almost bi-Lagrangian structure}\label{3}

From here on, we take a different point of view. We study the geometry on the total
space of the bundle of Weyl structures associated to a parabolic geometry from an
intrinsic point of view, using the relation to parabolic geometries and Weyl
structures as technical input. We shall see below that these structures become rather
exotic in the case of general gradings, so we will restrict to parabolic geometries
associated to $|1|$-gradings soon.

\subsection{The almost bi-Lagrangian structure and torsion freeness}\label{3.1}
Consider a parabolic geometry $(p:\Cal G\to M,\om)$ of some type $(G,P)$ and let
$\pi:A\to M$ the associated bundle of Weyl structures. As we have noted in
\cref{2.2}, the tangent bundle $TA$ decomposes as $L^-\oplus L^+$, where $L^-=\Cal
G\x_{G_0}\frak g_-$ and $L^+=\Cal G\x_{G_0}\frak p_+$. It is also well known that
$\frak g_-$ and $\frak p_+$ are dual as representations of $G_0$ via the restriction
of the Killing form of $\frak g$. Thus we obtain a non-degenerate pairing $B$ mapping
$L^-\x L^+$ to the trivial real line bundle $M\x\Bbb R$. This pairing can be extended
as either a skew symmetric or a symmetric bilinear bundle map on $TA$, thus defining
$\Om\in\Om^2(A)$ and $h\in\Ga(S^2T^*A)$. By construction, for each $y\in A$ both
values $\Om(y)$ and $h(y)$ are non-degenerate bilinear forms on $T_yA$ for which $L^+_y$ and $L^-_y$ are isotropic. The resulting
structure $(\Om,L^+,L^-)$ is called an \textit{almost bi-Lagrangian structure}.

In particular, $\Om\in\Om^2(A)$ is an almost symplectic structure and an obvious
first question is when this structure is symplectic, i.e.~when $d\Om=0$.
\begin{thm}\label{thm3.1}
   Let $(p:\Cal G\to M,\om)$ be a parabolic geometry of type $(G,P)$ and $\pi:A\to M$ its
  associated bundle of Weyl structures. Then the natural 2-form $\Om\in\Om^2(A)$ is closed if and only if $(G,P)$ corresponds to a $|1|$-grading
  and the Cartan geometry $(p:\Cal G\to M,\om)$ is torsion-free.
\end{thm}
\begin{proof}
  Let $D$ be the canonical connection on $TA$ from \cref{2.3}. Since $\Om$ is induced
  by a $G_0$-invariant pairing on $\frak g_-\oplus\frak p_+$, it satisfies
  $D\Om=0$. If $D$ were torsion-free, then $d\Om$ would coincide with the complete
  alternation of $D\Om$ and thus would vanish, too. In the presence of torsion, there
  still is a relation as follows. Expanding $D\Om=0$ by inserting vector fields
  $\xi,\eta,\ze\in\frak X(A)$, we obtain
$$
0=\xi\cdot\Om(\eta,\ze)-\Om(D_\xi\eta,\ze)-\Om(\eta,D_\xi\ze).
$$
Now one takes the sum of the right hand side over all cyclic permutations of the
arguments and uses skew symmetry of $\Om$ to bring all derivatives of vector fields
into the first component. Then one may use the definition of the torsion $\tau$ of
$D$ to rewrite $D_\xi\eta-D_\eta\xi$ as $[\xi,\eta]+\tau(\xi,\eta)$ and similarly for
other combinations of the fields. Then the terms in which one field differentiates
the value of $\Om$ together with the terms involving a Lie bracket add up to the
exterior derivative. One concludes that $D\Om=0$ implies
$$
d\Om(\xi,\eta,\ze)=\textstyle\sum_{\text{cycl}}\Om(\tau(\xi,\eta),\ze),
$$
where in the right hand side we have the sum over all cyclic permutations of the
arguments. Now let us assume that $(G,P)$ corresponds to a $|k|$-grading with
$k>1$. Then $L^-$ and $L^+$ decompose into direct sums of subbundles according to the
grading of $\frak g_-$ and $\frak p_+$, respectively. Now we take $\xi\in L^-$ of
degree $-1$, $\eta\in L^+$ of degree $i>1$ and $\ze\in L^+$ of degree $i-1$. Then by
\cref{thm2.7}, $\tau$ coincides with $\{\ ,\ \}$ on any two of these three
fields. The restriction of $d\Om$ to the subbundles corresponding to these three
degrees is induced by the trilinear map
$\frak g_{-1}\x \frak g_i\x\frak g_{i-1}\to\Bbb R$ given by
$(X,Y,Z)\mapsto \sum_{\text{cycl}}B([X,Y],Z)$, where $B$ denotes the Killing form of
$\frak g$. But $B([X,Y],Z)$ is already totally skew, so $d\Om=0$ would imply that
$B([X,Y],Z)=0$ for all elements of the given homogeneities. But non-degeneracy of $B$
shows that $B([X,Y],Z)=0$ for all $Z$ implies $[X,Y]=0$ while for $Y\in\frak g_i$,
the equation $[X,Y]=0$ for all $X\in\frak g_{-1}$ implies $Y=0$, see Proposition
3.1.2 in \cite{book}.

Thus we may assume from now on that $(G,P)$ corresponds to a $|1|$-grading. In this
case, the bracket $\{\ ,\ \}$ is identically zero, so by \cref{thm2.7}, $\tau$
vanishes upon insertion of one element from $L^+$. Hence we see that $d\Om$ vanishes
upon insertion of two elements of $L^+$. Decomposing $\La^3T^*A$ according to
$TA=L^-\oplus L^+$, the only potentially non-zero components of $d\Om$ thus are the
ones in $\La^3(L^-)^*$ and in $\La^2(L^-)^*\otimes (L^+)^*$.

Now if $\xi,\eta\in\Ga(L^-)$ and $\ze\in\Ga(L^+)$, then we simply obtain
$d\Om(\xi,\eta,\ze)=\Om(T(\xi,\eta),\ze)$, where $T$ is defined in \cref{def2.7}. Non-degeneracy of $\Om$ shows that this
vanishes for all $\xi,\eta,\ze$ if and only if $T=0$. This shows that vanishing of
$T$ is a necessary condition for $\Om$ being closed. In the case of a $|1|$-grading, the
pullback of $T$ along a Weyl structure as in \cref{prop2.7} is independent of
the Weyl structure and gives the torsion of the Cartan geometry $(p:\Cal G\to
M,\om)$.

To complete the proof, we thus have to show that (still in the case of a
$|1|$-grading) vanishing of $T$ implies that the component of $d\Om$ in
$\La^3(L^-)^*$ vanishes identically. As above, \cref{thm2.7} shows that this
component is given by the sum of $\Om(Y(\xi,\eta),\ze)$ over all cyclic permutations
of its arguments.  But by construction, this is simply the complete alternation of
$Y$, viewed as a section of $\La^2L^+\otimes L^+$ via the identification
$(L^-)^*\cong L^+$. In terms of the Cartan geometry $(p:\Cal G\to M)$, it thus
suffices to show that the component $\ka_+$ of the Cartan curvature in $\frak p_+$
always has trivial complete alternation.

To do this, we first observe that for a $|1|$-graded Lie algebra $\frak g$, the
subalgebra $\frak g_0$ always splits into its center $\frak z(\frak g_0)$, which has
dimension one, and a semisimple part $\frak g_0^{ss}$. For a torsion-free geometry,
the component $\ka_0$ of $\ka$ with values in $\frak g_0$ is the lowest non-vanishing
homogeneous component of $\ka$, which implies that its values have to lie in $\frak
g_0^{ss}$, compare with Theorem 4.1.1 in \cite{book}. Thus we conclude that, viewed
as a function $\Cal G\to \frak g$, $\ka$ has values in the subspace $\frak
g_0^{ss}\oplus\frak g_1$.

Now we can apply the Bianchi identity in the form of equation (1.25) in Proposition
1.5.9 of \cite{book}. This contains four terms, three of which are evaluations of the
function $\ka$ or its derivative along some vector field, so these have values in
$\frak g_0^{ss}\oplus\frak g_1$, too. Formulated in terms of functions, the Bianchi
identity thus implies that for $X_1,X_2,X_3$ the cyclic sum over the arguments of
$[X_1,\ka(\om^{-1}(X_2),\om^{-1}(X_3))]$ has trivial component in
$\frak z(\frak g_0)$. Now we can replace the $X_i$ by their components in $\frak g_-$
without changing the $\frak g_0$-component of
$[X_1,\ka(\om^{-1}(X_2),\om^{-1}(X_3))]$, which in addition depends only on the
$\frak g_1$-component of $\ka$. Now it is well known that for $X\in\frak g_{-1}$ and
$Z\in\frak g_1$, the component of $[X,Z]$ in $\frak z(\frak g_0)$ is a non-zero
multiple of $B(X,Z)$, where $B$ denotes the Killing form. But this exactly shows
that, up to a non-zero factor, the $\frak z(\frak g_0)$-component of
$\sum_{\text{cycl}}[X_1,\ka(\om^{-1}(X_2),\om^{-1}(X_3))]$ represents the action of
the complete alternation of $\ka_+$ on the three vector fields corresponding to the
$X_i$. Thus this complete alternation vanishes identically.
\end{proof}

\begin{remark}\label{rem3.1}
  (1) The failure of closedness of $\Om$ for $|k|$-gradings with $k>1$ can be
  described more precisely. The map $(X,Y,Z)\mapsto B([X,Y],Z)$ that shows up in the
  proof defines a $G$-invariant element in $\La^3\frak g^*$ and hence a bi-invariant
  $3$-form on $G$. Restricting this to $\La^3(\frak g_-\oplus\frak p_+)^*$, one
  obtains a $G_0$-invariant trilinear form, which is non-zero provided that
  $k>1$. This in turn induces a natural $3$-form on each manifold endowed with a
  Cartan Geometry of type $(G,G_0)$. On a bundle of Weyl structures, the proof of
  \cref{thm3.1} shows that this form always is a component of $d\Om$.

\smallskip

  (2) The parabolic geometries corresponding to $|1|$-gradings form a very
interesting class of structures. For a $|1|$-grading, the subalgebras $\frak g_-$ and
$\frak p_+$ become Abelian, whence the name ``Abelian parabolic geometries'' is
sometimes used for these structures. The classification of $|1|$-gradings of simple
Lie algebras is well known from the theory of Hermitian symmetric spaces, which
motivates the more common name ``AHS structures'' where AHS is shorthand for ``almost
Hermitian symmetric''.

  Suppose that $(G,P)$ corresponds to a $|1|$-grading on $\frak g$. As noted in
  \cref{2.1}, the underlying structure $p_0:\Cal G_0\to M$ of a Cartan geometry
  $(p:\Cal G\to M)$ simply becomes a reduction of the linear frame bundle of $M$ to
  the structure group $G_0\subset GL(\frak g_{-1})$. Thus AHS structures are a
  special class of G-structures, whose relevance is explained by the classification
  results by S.~Kobayashi and T.~Nagano in \cite{KN1}. They prove that these are the
  only structures for which the group acts irreducibly, and which have the property
  that any automorphism is determined by a finite jet in a point but not by the
  one-jet in a point. In fact, automorphisms are always determined by the two-jet in a
  point and the equivalent canonical Cartan geometry of type $(G,P)$ is the most
  effective description for these structures.

  The torsion-freeness condition that shows up in \cref{thm3.1} has a natural
  interpretation in the language of $G_0$-structures. As noted in the proof, the
  torsion $T$ associated to a Weyl structure in this case is independent of the Weyl
  structure. It turns out that this coincides with the intrinsic torsion of the
  $G_0$-structure (i.e.~the component of the torsion that is independent of the
  choice of connection). Thus torsion-freeness of the Cartan geometry corresponds to
  the usual notion of integrability in the language of $G_0$-structures.

\smallskip

(3) For some types of AHS-structures, torsion-freeness implies local
flatness. Locally flat structures can be equivalently be characterized as being
obtained from local charts with values in the homogeneous model $G/P$, for which the
transition functions are given by restrictions of left actions of elements of
$G$. This case anyway plays a very important role in the results we are going to prove,
so our results are also relevant to these types of AHS structures.
\end{remark}

\subsection{Local frames}\label{3.2}
From this point on, we restrict the discussion to torsion-free geometries of some
type $(G,P)$ that corresponds to a $|1|$-grading of $\frak g$, so that by
\cref{thm3.1} $\Om$ defines a symplectic structure on $A$. Recall from \cref{2.4}
that any vector field $\xi\in\frak X(M)$ determines a section $\tilde\xi\in\Ga(L^-)$,
and since $\frak g_-$ is a completely reducible representation in the $|1|$-graded
case, we get $D^+\tilde\xi=0$. Similarly, a one-form $\al\in\Om^1(M)$ defines a
section $\tilde\al\in\Ga(L^+)$ such that $D^+\tilde\al=0$. We further know that
$L^-\cong\pi^*TM$ and $L^+=\pi^*T^*M$. This implies that starting with local frames
for $TM$ and $T^*M$ defined on some open set $U\subset M$, the lifts form local
frames for $L^\pm$ defined on $\pi^{-1}(U)$, so together, these form a local frame
for $TA$. One may in particular use dual local frames for $TM$ and $T^*M$ in which
case the resulting local frame for $TA$ is nicely adapted to the almost bi-Lagrangian
structure and thus both to $\Om$ and to $h$. As a preparation for the following
computations, we next compute the Lie brackets of such sections.

\begin{prop}\label{prop3.2}
  Consider a torsion-free AHS structure $(p:\Cal G\to M,\om)$ and let $\pi:A\to M$ be
  the corresponding bundle of Weyl structures. Let $\xi,\eta\in\frak X(M)$ be vector
  fields and $\al,\be\in\Om^1(M)$ be one-forms on $M$ and consider the corresponding
  sections $\tilde\xi,\tilde\eta\in\Ga(L^-)$ and $\tilde\al,\tilde\be\in\Ga(L^+)$.

  Then for the Lie brackets on $A$, we get $[\tilde\al,\tilde\be]=0$ and
  $[\tilde\xi,\tilde\al]=D_{\tilde\xi}\tilde\al\in\Ga(L^+)$. Finally, the
  $L^-$-component of $[\tilde\xi,\tilde\eta]$ coincides with
  $\widetilde{[\xi,\eta]}$, while its $L^+$-component coincides with
  $-Y(\tilde\xi,\tilde\eta)$, see \cref{def2.7}.
\end{prop}
\begin{proof}
  By definition of the torsion
  \begin{equation}\label{tors}
    \tau(X,Z)=D_XZ-D_ZX-[X,Z]
  \end{equation}
  for all $X,Z\in\frak X(A)$. If at least one of the two fields is a section of
  $L^+$, then the left hand side of \eqref{tors} vanishes by
  \cref{thm2.7}. Moreover, all the sections coming from $M$ are parallel in
  $L^+$-directions. This immediately shows that $[\tilde\al,\tilde\be]=0$ and
  $0=D_{\tilde\xi}\tilde\al-[\tilde\xi,\tilde\al]$. In view of torsion-freeness,
  \cref{thm2.7} further tells us that
  $\tau(\tilde\xi,\tilde\eta)=Y(\tilde\xi,\tilde\eta)\in\Ga(L^+)$. Inserting
  $X=\tilde\xi$ and $Z=\tilde\eta$ into the right hand side of \eqref{tors}, the
  first two terms are sections of $L^-$, so the claim about the $L^+$-component
  of $[\tilde\xi,\tilde\eta]$ follows. Finally, since $\tilde\xi$ and $\tilde\eta$
  are lifts to $A$ of $\xi$ and $\eta$, the bracket
  $[\tilde\xi,\tilde\eta]\in\frak X(A)$ is a lift of $[\xi,\eta]$. Since
  $\widetilde{[\xi,\eta]}$ is the unique section of $L^-$ that projects onto
  $[\xi,\eta]$, it has to coincide with the $L^-$-component of that lift.
\end{proof}

In particular, we see that, while $L^+$ always defines an involutive distribution,
$L^-$ is only involutive if the curvature component $Y$ from \cref{def2.7} vanishes
identically. From the interpretation via the Cartan curvature, one easily concludes
that this is equivalent to local vanishing of the Cartan curvature. Thus our
structure is bi-Lagrangian (in the sense that both subbundles $L^\pm$ are integrable)
if and only if the initial parabolic geometry is locally flat.

\subsection{The canonical metric}\label{3.3}
We next study the pseudo-Riemannian metric $h$ induced on $A$. By
definition, the subbundles $L^\pm$ are isotropic for $h$, so this metric always has
split signature $(n,n)$, where $n=\dim(M)$. Our main next aim will be to prove that
the metric $h$ is always Einstein. As a first step in this direction, consider the
canonical connection $D$ and its curvature $\rho\in\Om^2(A,\End_0(TA))$ as described
in \cref{2.7}.

\begin{lemma}\label{lem3.3}
  The Ricci-type contraction of $\rho$ is a non-zero multiple of $h$.
\end{lemma}
\begin{proof}
  By \cref{thm2.7}, $\rho+\{\ ,\ \}_0$ vanishes upon insertion of one section of
  $L^+$ and coincides with $W$ on $\La^2(L^-)^*$. Decomposing $\La^2TA^*$ according
  to $TA=L^+\oplus L^-$, we conclude that the component of $\rho$ in $\La^2(L^+)^*$
  vanishes, its component in $(L^-)^*\otimes (L^+)^*$ is induced by $-\{\ ,\ \}_0$,
  and the component in $\La^2(L^-)^*$ is induced by $W$.  On the other hand,
  $\End_0(TA)=\Cal G\x_{G_0}\frak g_0$, so this is a subbundle of $((L^-)^*\otimes
  L^-)\oplus ((L^+)^*\otimes L^+)\subset TA^*\otimes TA$. Thus we conclude that the
  Ricci-type contraction of $\rho$ vanishes on $L^+\x L^+$, while its components on
  $L^-\x L^+$ and $L^-\x L^-$ are induced by the Ricci-type contractions of
  $-\{\ ,\ \}_0$ and $W$, respectively. By \cref{prop2.7}, the pullback of $W$ along
  any section $s:M\to A$ represents the Weyl curvature of the Weyl structure
  determined by $s$. In the torsion-free case, this is well known to have values in
  an irreducible representation of $G_0$ that occurs with multiplicity one in
  $\La^*\frak g_-^*\otimes\frak g$, which implies that any contraction of $W$
  vanishes identically.

  Hence we see that the Ricci type contraction of $\rho$ has values in
  $(L^-)^*\otimes (L^+)^*$ and is induced by the Ricci type contraction of
  $-\{\ ,\ \}_0$, so this is a natural bundle map, and we can compute it on the
  inducing representations.  Take a basis $\{e_i\}$ of $\frak g_{-1}$ and let
  $\{e^i\}$ be the dual basis of $\frak g_{-1}^*\cong\frak g_1$, which means that for
  the Killing form $B$, we get $B(e_i,e^j)=\delta_i^j$. Now we have to view the
  $\frak g_0$ component of the bracket $[\ ,\ ]$ in $\frak g$ as a map sending
  $(\frak g_-\oplus\frak p_+)^2$ to an endomorphism of $\frak g_-\oplus\frak p_+$ via
  the adjoint action. Hence for $X,Y\in\frak g_{-1}$ and $Z,W\in\frak g_1$, the Ricci
  type contraction sends $\binom{X}{Z}$ and $\binom{Y}{W}$ to
\[
\textstyle\sum_iB\big(\big[\big[\binom{X}{Z},\binom{e_i}{0}\big],\binom{Y}{W}\big],
\binom0{e^i})+\sum_i B\big(\big[\big[\binom{X}{Z},\binom{0}{e^i}\big],
\binom{Y}{W}\big],\binom{e_i}{0}\big).
\]
Expanding the first sum using invariance of the Killing form and the fact that $\frak
g_{-1}$ is Abelian, we obtain
\[
\textstyle\sum_iB([[Z,e_i],Y],e^i)=\sum_iB(Z,[e_i,[Y,e^i]])=\sum_iB(Z,[Y,[e_i,e^i]]), 
\]
and in the same way the second sum gives $\sum_iB([X,[e_i,e^i]],W)$. But the element
$\sum_i[e_i,e^i]\in\frak g_0$ is obtained from the identity map in $\frak
g_{-1}\otimes\frak g_{-1}^*$ via the isomorphism to $\frak g_{-1}\otimes\frak g_1$
and the bracket in $\frak g$. Since these both are $\frak g_0$-equivariant,
$\sum_i[e_i,e^i]$ is $\frak g_0$-invariant and thus contained in the center of $\frak
g_0$. In the $|1|$-graded case, this center is spanned by the grading element $E$. In
addition, $B(E,\textstyle\sum_i[e_i,e^i])=\sum_iB([E,e_i],e^i)=-\dim(\frak g_{-1})$,
so $\sum_i[e_i,e^i]$ is a non-zero multiple of $E$. Hence the whole contraction gives
a non-zero multiple of $B(Z,Y)+B(X,W)=h(\binom{X}{Z},\binom{Y}{W})$.
\end{proof}

Now by construction, the canonical connection $D$ satisfies $Dh=0$, so $D$ is metric
for $h$. This implies that the Levi-Civita connection $\nabla$ of $h$ can be computed
from $D$ and its torsion $\tau$. Indeed, we claim that for
$\xi,\eta,\ze\in\frak X(A)$, $h(\nabla_\xi\eta,\ze)$ is given by
\begin{equation}\label{L-C}
  h(D_\xi\eta,\ze)-\tfrac12h(\tau(\xi,\eta),\ze)+\tfrac12h(\tau(\xi,\ze),\eta)+ 
  \tfrac12h(\tau(\eta,\ze),\xi). 
\end{equation}
This evidently defines a linear connection $\nabla$ on $TA$. Moreover, the last three
terms in \eqref{L-C} are visibly skew symmetric in $\eta$ and $\ze$, whence the fact
that $D$ is metric with respect to $h$ implies that $\nabla$ is metric with respect
to $h$, too. On the other hand, since the last two terms in \eqref{L-C} are symmetric
in $\xi$ and $\eta$, and $\tau$ is the torsion of $D$, one immediately verifies that
$\nabla$ is torsion-free. Let us write $C\in\Ga(\otimes^2T^*A\otimes TA)$ for the
contorsion tensor between $\nabla$ and $D$, so $C(\xi,\eta)=\nabla_\xi\eta-D_\xi\eta$
and the last three terms in \eqref{L-C} explicitly express
$h(C(\xi,\eta),\zeta)$. Using this, we prove the following result

\begin{thm}\label{thm3.3}
  For any torsion-free AHS structure, the pseudo-Riemannian metric $h$ induced by the
  canonical almost bi-Lagrangian structure on the bundle $A$ of Weyl structures is an
  Einstein metric with non-zero scalar curvature.
\end{thm}
\begin{proof}
  \cref{thm2.7} in the torsion-free case shows that $\tau$ vanishes upon
  insertion of one section of $L^+$ and has values in $L^+$. Thus equation
  \eqref{L-C} shows that $h(C(\xi,\eta),\ze)$ vanishes if one of the three fields is
  a section of $L^+$. This shows that the only non-zero component of $C$ is the one
  mapping $L^-\x L^-$ to $L^+$. Now it is standard how to compute the curvature of
  $\nabla$ from $C$ and the curvature $\rho$ of $D$ via differentiating the equation
  defining $C$. The result contains terms in which $C$ is differentiated as well as
  terms in which values of $C$ are inserted into $C$. From the form of $C$ we have
  just deduced, it follows that the latter terms vanish identically.

  Using this, one computes that for $\xi,\eta,\zeta\in\frak X(A)$ the difference
  $R(\xi,\eta)(\zeta)-\rho(\xi,\eta)(\zeta)$ is given by
  \begin{equation}
    \label{C-term}
    D_\xi(C(\eta,\zeta))-D_\eta(C(\xi,\zeta))+C(\xi,D_\eta\zeta)-C(\eta,D_\xi\zeta)-C([\xi,\eta],\zeta). 
  \end{equation}
  (This is just the covariant exterior derivative of $C$ with respect to $D$
  evaluated on $\xi$ and $\eta$ and then applied to $\ze$.)  In view of
  \cref{lem3.3}, it suffices to prove that the Ricci-type contraction of this
  expression vanishes. To compute this contraction, we leave $\xi$ and $\zeta$ as
  entries, insert the elements of a local frame of $TA$ for $\eta$ and hook the
  result into $h$ together with the elements of the dual frame. First of all,
  \eqref{C-term} visibly vanishes for $\zeta\in\Ga(L^+)$. If we insert for $\eta$ an
  element of a frame for $L^-$, then the element of the dual frame will sit in
  $L^+$. Since $C$ has values in $L^+$, these summands do not contribute to the
  contraction. Thus we only have to take into account the case where we insert
  elements of a frame for $L^+$ for $\eta$, and then the first and fourth term of
  \eqref{C-term} visibly vanish. The remaining three terms vanish if $\xi$ is a
  section of $L^+$, so what we have to compute is
\[
\textstyle\sum_ih\left(-D^+_{e^i}(C(\xi,\zeta))+C(\xi,D^+_{e^i}\zeta)-C([\xi,e^i],\zeta),e_i\right) 
\]
for a smooth local frame $\{e^i\}$ for $L^+$ with dual frame $\{e_i\}$ for $L^-$ and
local sections $\xi,\zeta\in\Ga(L^-)$. Now we can take $\xi$ and $\zeta$ and the
local frames to be obtained from vector fields respectively one-forms on $M$. Then
$D^+_{e^i}\zeta=0$, while $[\xi,e^i]\in\Ga(L^+)$ by \cref{prop3.2} and
thus $C([\xi,e^i],\zeta)=0$.

Thus we are left with computing $\sum_ih(D^+_{e^i}(C(\xi,\zeta)),e_i)$ with the
frames, $\xi$ and $\zeta$ all coming from $M$. In particular, $e_i$ is parallel for
$D^+$ so since $D$ is metric for $h$, we may rewrite this as $\sum_ie^i\cdot
h(C(\xi,\zeta),e_i)$. We can then insert the formula for $h(C(\xi,\zeta),e_i)$
resulting from \eqref{L-C}, taking into account that on entries from $L^-$ the
torsion $\tau$ is determined by the tensor $Y$ from \cref{def2.7}. Viewing $Y$ as a
section of $\La^2(L^-)^*\otimes (L^-)^*$, this leads to
\[
\textstyle\sum_ie^i\cdot h(C(\xi,\zeta),e_i)=\tfrac12\sum_ie^i\cdot
(-Y(\xi,\zeta,e_i)+Y(\xi,e_i,\zeta)+Y(\zeta,e_i,\xi)).
\]
From the proof of \cref{thm3.1}, we know that the complete alternation of $Y$
vanishes, which allows us to rewrite this as $\sum_ie^i\cdot Y(\zeta,e_i,\xi)$.  Under
the standing assumption that all sections come from $M$, they are parallel for $D^+$,
so we can complete the proof by showing that $\sum_i(D^+_{e^i}Y)(\zeta,e_i,\xi)=0$.

Now by definition, $Y$ is a component of the Cartan curvature, which descends to a
well defined section of the bundle $\La^2T^*M\otimes\Cal AM$, where
$\Cal AM=\Cal G\x_P\frak g$. By torsion freeness, the full Cartan curvature has the
form $(0,W,Y)$ with respect to the decomposition
$\frak g=\frak g_{-1}\oplus\frak g_0\oplus\frak g_1$. Hence by \cref{thm2.4},
we get
\[
D^+_\ph(0,W,Y)=(0,D^+_\ph W,D^+_\ph,Y)=-\ph\bullet (0,W,Y),
\] 
and the action $\bullet$ is induced by the Lie bracket on $\frak g$. Since this
bracket vanishes on $\frak g_1\x\frak g_1$ and defines the action of $\frak g_0$ on
$\frak g_{\pm 1}$, we conclude that $D^+W=0$ and $D^+_\ph Y=W(\ph)$, where we view
$W$ as an section of $\La^2(L^--)^*\otimes\End(L^+)\cong
\La^2(L^--)^*\otimes(L^-)^*\otimes (L^+)^*$ in the right hand side. But this implies
that $\sum_i(D^+_{e^i}Y)(\zeta,e_i,\xi)$ is given by evaluating a trace of $W$ on
$\zeta$ and $\xi$ and thus vanishes.
\end{proof}

\begin{remark}
In the special case of a projective structure on a surface $\Sigma$, the resulting Einstein metric on the four-manifold $A$ is also anti-self-dual, see~\cite{Dunajski-Mettler} and also~\cite{MR3210600}.
\end{remark}

\begin{remark}
The automorphisms of a projective structure on a smooth manifold $M$ lift to become isometric symplectomorphisms of $(A,\Omega,h)$, see~\cite{Dunajski-Mettler}. For background about automorphisms of a parabolic geometry, see~\cite{MR2545240,book,Frances}.
\end{remark}

\begin{remark}
A pair $(h,\Omega)$, consisting of a split-signature metric $h$ and a symplectic form $\Omega$ that are related by an endomorphism which squares to become the identity map, is also known as an \textit{almost para-K\"ahler structure}. Here, following~\cite{MR1824987}, we refer to such a pair, or rather its associated triple $(\Omega,L_+,L_-)$, as an almost bi-Lagrangian structure. 
\end{remark}

\subsection{Geometry of Weyl structures}\label{3.4}
Viewed as a section of $\pi:A\to M$, any Weyl structure defines an embedding of $M$
into $A$, and we can now study this embedding via submanifold geometry related to the
almost bi-Lagrangian structure. In particular, we can pull back the
two-form $\Om$ and the pseudo-Riemannian metric to $M$ along $s$, and this naturally
leads to the following definitions.

\begin{definition}
  Let $(p:\Cal G\to M,\om)$ be a torsion-free AHS structure and let $s:M\to A$ be a
  smooth section. 

(1) The Weyl structure corresponding to $s$ is called \textit{Lagrangian} if and only
if $s^*\Om=0$ and thus $s(M)\subset A$ is a Lagrangian submanifold. 

(2) The Weyl structure corresponding to $s$ is called \textit{non-degenerate} if and
only if $s^*h\in\Ga(S^2T^*M)$ is non-degenerate and thus defines a pseudo-Riemannian
metric on $M$. 
\end{definition}

These properties can easily be characterized in terms of the Rho tensor. 

\begin{prop}\label{prop3.4}
  A Weyl structure is Lagrangian if and only if its Rho tensor is symmetric and
  it is non-degenerate if and only if the symmetric part of its Rho tensor is
  non-degenerate.
\end{prop}
\begin{proof}
  For a point $x\in M$ and a tangent vector $\xi\in T_xM$ consider
  $T_xs\cdot\xi\in T_{s(x)}(A)$. Since this is a lift of $\xi$, its $L^-$-component
  has to coincide with $\tilde\xi(s(x))$. On the other hand, by
  \cref{prop2.6}, the $L^+$-component of $T_xs\cdot\xi$ corresponds to
  $\Rho(x)(\xi)\in T^*_xM$. Pulling back the pairing between $L^-$ and $L^+$, one
  thus obtains the map $(\xi,\eta)\mapsto \Rho(x)(\eta)(\xi)$ and thus the result
  follows from the definitions of $\Om$ and $h$.
\end{proof}

\begin{remark}
(1) For any type of parabolic geometry, there are natural line bundles called
  \textit{bundles of scales}, an example in the AHS-case is provided by the
  bundle $\Cal EM$ in \cref{thm4.1} below. If $\Cal EM$ is any bundle of scales
  on $M$, then mapping a Weyl structure to the induced Weyl connection on $\Cal EM$
  induces a bijective correspondence between Weyl structures and linear connections
  on $\Cal EM$. Fixing $\Cal EM$, one calls a Weyl structure \textit{closed} if the
  corresponding linear connection on $\Cal EM$ is flat and \textit{exact} if in
  addition there is a global parallel section of $\Cal EM$, see \cite{Weyl} and
  Section 5.1 of \cite{book}. For general types of geometries there is a larger
  freedom of choice of bundles of scales, but for AHS-structures all bundles of
  scales lead to the same subclasses of closed and exact Weyl structures.

  Together with the general theory of Weyl structures, \cref{prop3.4}
  implies that, on a torsion-free AHS-structure, a Weyl structure is Lagrangian if
  and only if it is closed. This follows from the relation between curvature and
  torsion of a Weyl connection and the Cartan curvature as discussed in Example 5.2.3
  of \cite{book} in the setting of AHS structures. By Theorem 5.2.3 of that
  reference, the curvature $R\in\Om^2(M,\Cal G\x_g\frak g_0)$ of the Weyl connection
  corresponding to $s\in\Ga(A)$ can be computed as $s^*W+\partial s^*\Rho$ for the
  quantities from \cref{prop2.7} and a certain natural bundle map
  $\partial$. For the action on a bundle of scales, the component of $R$ with values
  in the center $\frak z(\frak g_0)$ is relevant. For a torsion free geometry, $W$ is
  the lowest non-zero component of the Cartan curvature and hence by general results
  has values in an irreducible subrepresentation which cannot meet this center. Hence
  the curvature of the Weyl connection is only induced by $\partial s^*\Rho$, and the
  component of this in $\frak z(\frak g_0)$ is immediately seen to be the skew part
  of the Rho tensor up to a non-zero factor.

  This result nicely corresponds to the fact that for the canonical symplectic
  structure on any cotangent bundle $T^*N$, the image of a one-form $\al\in\Om^1(N)$
  in $T^*N$ is a Lagrangian submanifold if and only if $d\al=0$.

  \smallskip

(2) In the case of an AHS-structure, the cotangent bundle $T^*M$ coincides with the
  associated graded bundle, so \cref{prop2.2a} shows that a Weyl structure
  $s$ determines a diffeomorphism $\phi_s:T^*M\to A$. Now we can use this to pull
  back the geometric structures on $A$ to $T^*M$ and in particular, in the
  torsion-free case, compare the pullback of the symplectic form $\Om\in\Om^2(A)$ to the
  canonical symplectic structure on $T^*M$. Recall that the diffeomorphism $\phi_s$
  is induced by $\Phi_s:\Cal G_0\x\frak p_+\to\Cal G$,
  $\Phi_s(u_0,Z)=\si(u_0)\cdot\exp(Z)$, where $\si:\Cal G_0\to\Cal G$ is the
  equivariant section determined by $s$.

  Equivariancy of the Cartan connection $\om\in\Om^1(\Cal G,\frak g)$ then implies
  that the pullback $\Phi_s^*\om$ can be easily expressed explicitly in terms of
  $\si^*\om$. Denoting by $q:\Cal G_0\x\frak p_+\to T^*M$ the canonical projection,
  the definition of $\Om$ in \cref{3.1} shows that
  $q^*\phi_s^*\Om=\Phi_s^*\Om$ sends tangent vectors $\xi,\eta$ to the alternation of
  the pairing between $(\Phi_s^*\om)_-(\xi)\in\frak g_{-1}$ and
  $(\Phi_s^*\om)_+(\eta)\in\frak g_1$. On the other hand, it is easy to explicitly
  describe $q^*\al\in\Om^1(\Cal G_0\x\frak p_+)$, where $\al\in\Om^1(T^*M)$ is the
  canonical one-form. From this, one can explicitly compute the pullback $-q^*d\al$
  of the canonical symplectic form on $T^*M$ and show that it equals the sum of
  $\Phi_s^*\Om$ and the pullback of the alternation of the Rho-tensor. In particular,
  generalizing a result from \cite{Mettler:MinLag} in the projective case, we
  conclude that $\phi_s:T^*M\to A$ is a symplectomorphism if and only if the
  Weyl-structure $s$ is Lagrangian. Indeed, it turns out that also the
  split-signature metric $\phi_s^*h$ on $T^*M$ can be computed explicitly in terms of
  the underlying AHS-structure. All this will be taken up in more detail elsewhere.
\end{remark}

\smallskip

To start the geometric study of Lagrangian Weyl structures, we can characterize when
$s$ has the property that the submanifold $s(M)\subset A$ is totally geodesic.

\begin{thm}\label{thm3.4.1}
Let $(p:\Cal G\to M,\om)$ be a torsion-free AHS structure and $\pi:A\to M$ its bundle
of Weyl structures. Let $s:M\to A$ be a smooth section corresponding to a Lagrangian
Weyl structure, let $\nabla^s$ denote the corresponding Weyl connections and $\Rho^s$
the corresponding Rho-tensor. Then the following conditions are equivalent:

(1) The submanifold $s(M)\subset A$ is totally geodesic for the canonical
connection $D$.

(2) The  submanifold $s(M)\subset A$ is totally geodesic for the Levi-Civita
connection of $h$.

(3) $\nabla^s\Rho^s=0$
\end{thm}
\begin{proof}
  We will use abstract index notation to carry out the computations and denote the
  Rho tensor of $s$ just by $\Rho$, so this has the form $\Rho_{ij}$ and is symmetric
  by assumption. Since for each $y\in A$ and $x:=\pi(y)\in M$, we can identify
  $L^-_y$ with $T_xM$ and $L^+_y$ with $T^*_xM$, we can use the index notation also
  on $A$, but here tangent vectors have the form $(\xi^i,\al_j)$. In this language
  the proof of \cref{prop3.4} shows that for $x\in M$ the tangent space
  $T_{s(x)}s(M)$ consists of all pairs of the form $(\xi^i,\Rho_{ja}\xi^a)$. The
  condition that $s(M)$ is totally geodesic with respect to $D$ means that for vector
  field $(\xi,\al)$ on $A$ that is tangent to $s(M)$ along $s(M)$, also the covariant
  derivative in directions tangent to $s(M)$ is tangent to $s(M)$.

  In particular, for a vector field $\eta\in\frak X(M)$, we know from above that
  $(\widetilde{\eta^j},\widetilde{\Rho_{kb}\eta^b})\in\frak X(A)$ is tangent to
  $s(M)$ along $s(M)$. Since these fields are parallel for $D^+$, we see that $s(M)$
  is totally geodesic for $D$ if and only if all derivatives with respect to $D$ of
  that field are tangent to $s(M)$ along $s(M)$. But this can be checked by pulling
  back the components of $D(\widetilde{\eta^j},\widetilde{\Rho_{kb}\eta^b})$ along
  $s$, which by \cref{thm2.5} leads to $\nabla^s_i\eta^j$ and
\begin{equation}\label{L+-comp}
\nabla^s_i\Rho_{ka}\eta^a=\eta^a\nabla^s_i\Rho_{ka}+\Rho_{ka}\nabla^s_i\eta^a,
\end{equation}
  respectively. So evidently, the result is tangent to $s(M)$ if and only if
  $\eta^a\nabla^s_i\Rho_{ka}=0$ and since this has to hold for each $\eta$, we
  conclude that (1) is equivalent to (3).

  To deal with (2), we use the information on the contorsion tensor $C$ from
  \cref{3.3}. As observed in the proof of \cref{thm3.3}, the only non-zero component
  of $C$ maps $L^-\x L^-$ to $L^+$. This means that
  $(\widetilde{\eta^j},\widetilde{\Rho_{kb}\eta^b})$ is also parallel in
  $L^+$-directions for the Levi-Civita connection, so as above, we can use the pull
  back of the full derivative along $s$ and the result has to be tangent to $s(M)$. We write the pullback of $C$ along $s$ as $C_{ijk}$ using the convention
  that $C(\xi,\eta)_k=\xi^i\eta^jC_{ijk}$. Now formula \eqref{L-C} from \cref{3.3}
  expresses $C$ in terms of the torsion $\tau$ of $D$ (with $h$ just playing the role
  of identifying $L^+$ with the dual of $L^-$) and we know that the torsion
  corresponds to the Cartan curvature quantity $Y$, see \cref{thm2.7}. Writing the
  pullback of this along $s$ in abstract index notion as $Y_{ijk}$, we conclude form
  formula \eqref{L-C} that $C_{ijk}=\tfrac12(-Y_{ijk}+Y_{ikj}+Y_{jki})$. The pullback
  of the derivative along $s$ again has first component $\nabla^s_i\eta^j$ but for
  the second component, we have to add $C_{iak}\eta^a$ to the right hand side of
  \eqref{L+-comp}.

  But it is a well known fact (see Theorem 5.2.3 of \cite{book}) that the pullback of
  $Y$ along $s$ is given by the covariant exterior derivative of the Rho-tensor of
  $s$ and since $\nabla^s$ is torsion-free, this is expressed in abstract index
  notation as $Y_{ijk}=\nabla^s_i\Rho_{jk}-\nabla^s_j\Rho_{ik}$. Inserting this into
  the formula for $C_{ijk}$, we immediately conclude that we have to add
  $\eta^a\nabla^s_a\Rho_{ik}-\eta^a\nabla^s_k\Rho_{ia}$ to \eqref{L+-comp}. As above,
  this implies that (2) is equivalent to
  \begin{equation}\label{LC-geod}
  \nabla^s_i\Rho_{ka}+\nabla^s_a\Rho_{ik}-\nabla^s_k\Rho_{ia}=0. 
  \end{equation}
  Of course, \eqref{LC-geod} is satisfied if $\nabla^s\Rho=0$. Conversely, if
  \eqref{LC-geod} holds, then summing over all cyclic permutations of the indices
  shows that the total symmetrization of $\nabla^s\Rho$ has to vanish. But
  subtracting three times this total symmetrization from the left hand side of
  \eqref{LC-geod}, one obtains $-2\nabla^s_k\Rho_{ia}$, so this has to vanish, too.
\end{proof}

The result of \cref{thm3.4.1} is particularly interesting if $s$ is non-degenerate. By
\cref{prop3.4} this implies that $\Rho^s$ defines a pseudo-Riemannian metric on $M$
and since $\nabla^s$ is torsion-free, $\nabla^s\Rho^s=0$ implies that $\nabla^s$ is
the Levi-Civita connection of $\Rho^s$. On the other hand, $\Rho^s$ is always related
to the Ricci-type contraction of the curvature of $\nabla^s$, see Section 4.1.1 of
\cite{book}. In particular, for projective structures, symmetry of $\Rho^s$ implies
that it is a non-zero multiple of the Ricci curvature of $\nabla^s$, see \cite{BEG},
so in this case $\Rho^s$ defines an Einstein metric on $M$. The condition that a
projective structure contains the Levi-Civita connection of an Einstein metric can be
expressed as a reduction of projective holonomy, see \cite{Einstein} and
\cite{hol-red}.

\medskip

For a non-degenerate Lagrangian Weyl structure $s$, there is a well defined second
fundamental form of $s(M)$ with respect to any linear connection on $TA$ which is
metric for $h$. Extending the result of \cref{thm3.4.1} in this case, we can next
explicitly compute the second fundamental forms for $D$ and for the Levi-Civita
connection. To formulate the result, we use abstract index notation as in the proof
of \cref{thm3.4.1}. 

Fix the section $s:M\to A$ corresponding to a
non-degenerate, Lagrangian Weyl structure. By non-degeneracy, the Rho tensor
$\Rho_{ij}$ of $s$ admits an inverse $\Rho^{ij}\in\Ga(S^2TM)$ which is characterized
by $\Rho^{ij}\Rho_{jk}=\delta^i_k$. In the proof of \cref{thm3.4.1}, we have seen
that $T_{s(x)}s(M)$ is spanned by all elements of the form
$(\eta^i,\Rho_{ka}\eta^a)$ with $\eta^i\in T_xM$. The definition of $h$ readily
implies that the normal space $T^\perp_{s(x)}s(M)$ consists of all pairs of the form
$(\eta^i,-\Rho_{jk}\eta^k)$, so we can identify both the tangent and the normal space
in $s(x)$ with $T_xM$ via projection to the first component. Correspondingly, the
second fundamental form of $s(M)$ (with respect to any connection on $TA$ which is
metric for $h$) can be viewed as a $\binom12$-tensor field on $M$. We denote the
second fundamental form of $s : M \to (A,h)$ with respect to $D$ by
$\mathrm{II}^s_{D}$ and with respect to the Levi-Civita connection of $h$ by
$\mathrm{II}^s_h$.

\begin{thm}\label{thm3.4.2}
Let $(p:\Cal G\to M,\om)$ be a torsion-free AHS-structure with bundle of Weyl
structures $\pi:A\to M$. Let $s:M\to A$ be a non-degenerate, Lagrangian Weyl
structure with Weyl connection $\nabla^s$ and Rho-tensor $\Rho\in\Ga(S^2T^*M)$, then
\[
\mathrm{II}^s_{D}=-\tfrac12 \Rho^{ka}\nabla^s_i\Rho_{ja} \qquad\text{and}\qquad 
\mathrm{II}^s_{h}=-\tfrac12
\Rho^{ka}(\nabla^s_i\Rho_{ja}+\nabla^s_j\Rho_{ia}-\nabla^s_a\Rho_{ij}).\] 
\end{thm}
\begin{proof}
  We only have to project the derivatives computed in the proof of \cref{thm3.4.1}
  to the normal space. From the description of tangent and normal spaces, it follows
  readily that projecting $(\xi^i,\al_j)$ to the normal space and taking the
  $L^-$-component of the result, one obtains $\tfrac12(\xi^i-\Rho^{ia}\al_a)$. Using
  this, the formulae follow directly from the the proof of \cref{thm3.4.1}.
\end{proof}

\section{Relations to non-linear invariant PDE}\label{4}
We conclude this article by discussing a relation between the geometry on the bundle
$A$ of Weyl structures and non-linear invariant PDE associated to AHS structures. We
will mainly consider the prototypical example of a projectively invariant PDE of
Monge-Amp\`ere type. We briefly discuss analogs of this and other invariant
non-linear PDE for specific types of AHS structures, but this will be taken up in detail
elsewhere.

\subsection{A tractorial description of \texorpdfstring{$A$}{A}}\label{4.1}
We start by deriving an alternative description of the bundle $A\to M$ of Weyl
structures associated to an AHS structure $(p:\Cal G\to M,\om)$ based on tractor
bundles. As mentioned in \cref{2.4}, these are bundles associated to representations
of $P$ that are restrictions of representations of $G$. An important feature of these
bundles is that they inherit canonical linear connections from the Cartan connection
$\om$. Together with some algebraic ingredients, these form the basis for the
machinery of BGG sequences that was developed in \cite{CSS-BGG} and
\cite{Calderbank-Diemer}, which will provide input to some of the further
developments. We have also met the canonical invariant filtration on representations
of $P$ in \cref{2.4} and the corresponding filtration of associated bundles by
smooth subbundles. For representations of $G$, these admit a simpler description,
that we derive first. In most of the examples we need below, this description is
rather obvious, so readers not interested in representation theory aspects can safely
skip the proof of this result.

Recall that for any $|k|$-grading on $\mathfrak g$, there is a unique grading element $E$, 
such that for $i=-k,\dots,k$ the subspace $\mathfrak g_i$ is the eigenspace with eigenvalue $i$ for
the adjoint action of $E$. In particular, $E$ has to lie in the center of the subalgebra $\mathfrak g_0$.
In the case of a $|1|$-grading, this center has dimension $1$ and thus is spanned by $E$. 

\begin{lemma}\label{lem4.1}
  Consider a Lie group $G$ with simple Lie algebra $\frak g$ that is endowed with a
  $|1|$-grading with grading element $E$, and let $G_0\subset P\subset G$ be subgroups associated to this
  grading. Let $\Bbb V$ be a representation of $G$ which is irreducible as a
  representation of $\frak g$. Then there is a $G_0$-invariant decomposition $\Bbb
  V=\Bbb V_0\oplus\dots\oplus\Bbb V_N$ such that
  \begin{itemize}
    \item Each $\Bbb V_j$ is an eigenspace of $E$. 
    \item For $i\in\{-1,0,1\}$ and each $j$, we have $\frak g_i\cdot\Bbb
      V_j\subset\Bbb V_{i+j}$.
    \item For each $j>0$, restriction of the representation defines a surjection
      $\frak g_1\otimes\Bbb V_{j-1}\to\Bbb V_j$.
    \item The canonical $P$-invariant filtration on $\Bbb V$ is given by $\Bbb
      V^j=\oplus_{\ell\geq j}\Bbb V_\ell$. 
  \end{itemize}
\end{lemma}
\begin{proof}
Note first that the $|1|$-grading of $\frak g$ induces a $|1|$-grading on the
complexification $\frak g^{\Bbb C}$ of $\frak g$, which has the same grading element
$E$ as $\frak g$. It is also well known that there is a Cartan subalgebra of $\frak
g^{\Bbb C}$ that contains $E$. Complexifying $\Bbb V$ and $\frak g$ if necessary and
then passing back to the $E$-invariant subspace $\Bbb V$, we conclude that $E$ acts
diagonalizably on $\Bbb V$. Denoting the $\la$-eigenspace for $E$ in $\Bbb V$ by
$\Bbb V_\la$, it follows readily that $\frak g_i\cdot\Bbb V_\la\subset\Bbb V_{\la+i}$
for $i\in\{-1,0,1\}$. Now take an eigenvalue $\la_0$ with minimal real part, let $N$
be the smallest positive integer such that $\la_0+N+1$ is not an eigenvalue of $E$
and put $\Bbb V_j:=\Bbb V_{\la_0+j}$ for $j=0,\dots,N$. Then, by construction, $\frak
g_{-1}$ acts trivially on $\Bbb V_0$, $\frak g_1$ acts trivially on $\Bbb V_N$, and
each $\Bbb V_j$ is $\frak g_0$-invariant. This shows that $\Bbb
V_0\oplus\dots\oplus\Bbb V_N$ is $\frak g$-invariant and hence has to coincide with
$\Bbb V$ by irreducibility.

By definition, the adjoint action of each element $g_0\in G_0$ preserves the grading
of $\frak g$, which easily implies that $\Ad(g_0)(E)$ acts on $\frak g_i$ by
multiplication by $i$ for $i\in\{-1,0,1\}$. This means that $\Ad(g_0)(E)-E$ lies in
the center of $\frak g$, so $\Ad(g_0)(E)=E$. But then for $v\in\Bbb V$, we can
compute $E\cdot g_0\cdot v$ as $\Ad(g_0)(E)\cdot g_0\cdot v=g_0\cdot E\cdot v$. This
shows that each $\Bbb V_j$ is $G_0$-invariant and it only remains to prove the last
two claimed properties of the decomposition.

We put $\tilde{\Bbb V}_0:=\Bbb V_0$ and for $j>0$, we inductively define $\tilde{\Bbb
  V}_j$ as the image of the map $\frak g_1\otimes\tilde{\Bbb V}_{j-1}\to\Bbb
V_j$. Then, by construction, each $\tilde{\Bbb V}_j$ is a $\frak g_0$-invariant
subspace of $\Bbb V_j$, so $\tilde{\Bbb V}:=\oplus_{j=0}^N\tilde{\Bbb V}_j\subset\Bbb
V$ is invariant under the actions of $\frak g_0$ and $\frak g_1$. But for $X\in\frak
g_{-1}$ and $Z\in\frak g_1$, we have $[X,Z]\in\frak g_0$ and for $v\in\Bbb V$ we get
\[
X\cdot Z\cdot v=Z\cdot X\cdot v+[Z,X]\cdot v. 
\] This inductively shows that $\tilde{\Bbb V}$ is invariant under the action of
$\frak g_{-1}$. Thus it is $\frak g$-invariant and hence has to coincide with $\Bbb
V$ by irreducibility, so it remains to verify the claimed description of the
canonical $P$-invariant filtration.

To do this, we first claim that an element $v\in\Bbb V$ such that $Z\cdot v=0$ for
all $Z\in\frak g_1$ has to be contained in $\Bbb V_N$. It suffices to prove this for
the complexification, so we may assume that both $\frak g$ and $\Bbb V$ are complex
and so there is a highest weight vector $v_0\in\Bbb V$ which is unique up to scale by
irreducibility. It is then well known that $\Bbb V$ is spanned by vectors obtained
from $v_0$ by the iterated action of elements in negative root spaces of $\frak
g$. Since on such elements $E$ has non-positive eigenvalues, we conclude that
$v_0\in\Bbb V_N$. Now assume that for some $j<N$, the space $W:=\{v\in\Bbb V_j:\frak
g_1\cdot v=\{0\}\}$ is non-trivial. Then, by construction, this is a $\frak
g_0$-invariant subspace of $\Bbb V_j$ on which the center of $\frak g_0$ acts by a
scalar, so it must contain a vector that is annihilated by all elements in positive
root spaces of $\frak g_0$. But since any positive root space of $\frak g$ either is
a positive root space of $\frak g_0$ or is contained in $\frak g_1$, this has to be a
highest weight vector for $\frak g$, which contradicts uniqueness of $v_0$ up to
scale.

Having proved the claim, we can first interpret it as showing that $\Bbb V^N=\Bbb
V_N$. From this the description of the $P$-invariant filtration follows by backwards
induction: Suppose that that we have shown that $\Bbb V^{N-j}=\Bbb
V_{N-j}\oplus\dots\oplus\Bbb V_N$ and let $w\in\Bbb V$ be such that for all
$Z\in\frak g_1$, we have $Z\cdot w\in\Bbb V^{N-j}$. Decomposing $w=w_0+\dots+w_N$, we
conclude that for all $i<N-j-1$ we must have $Z\cdot w_i=0$ and hence $w\in\Bbb
V_{N-j-1}\oplus\dots\oplus\Bbb V_N$. Together with the obvious fact that
$\oplus_{\ell\geq N-j-1}\Bbb V_\ell\subset\Bbb V^{N-j-1}$, this implies the
description of the $P$-invariant filtration.
\end{proof}

Using this, we can now prove an alternative description of the bundle of Weyl
structures that, as we shall see in the examples below, generalizes the construction
of \cite{Dunajski-Mettler}.

\begin{thm}\label{thm4.1}
Suppose that $(G,P)$ corresponds to a $|1|$-grading of the Lie algebra $\frak g$ of
$G$. Let $\Bbb V$ be a representation of $G$, which is non-trivial and irreducible as
a representation of $\frak g$, with natural $P$-invariant filtration $\{\Bbb
V^j:j=0,\dots,N\}$ such that $\Bbb V/\Bbb V^1$ has real dimension $1$. For a
parabolic geometry $(p:\Cal G\to M,\om)$ of type $(G,P)$ let $\Cal VM$ be the tractor
bundle induced by $\Bbb V$, $\Cal V^jM$ the subbundle corresponding to $\Bbb V^j$,
and define $\Cal EM$ to be the real line bundle $\Cal VM/\Cal V^1M$.

Then the bundle $A\to M$ of Weyl structures can be naturally identified with the open
subbundle in the projectivization $\Cal P(\Cal VM/\Cal V^2M)$ formed by all lines
that are transversal to the subbundle $\Cal V^1M/\Cal V^2M$ of hyperplanes. This in
turn leads to an identification of $A\to M$ with the bundle of all linear connections
on the line bundle $\Cal EM\to M$.
\end{thm}
\begin{proof}
By assumption $\Bbb V^1\subset\Bbb V$ is a $P$-invariant hyperplane, so this descends
to a $P$-invariant hyperplane $\Bbb V^1/\Bbb V^2$ in $\Bbb V/\Bbb V^2$. Passing to
the projectivization, the complement of this hyperplane is a $P$-invariant open
subset $U\subset\Cal P(\Bbb V/\Bbb V^2)$ and hence defines a natural open subbundle in the
associated bundle $\Cal P(\Cal VM/\Cal V^2M)$.

Now take the decomposition $\Bbb V=\oplus_{j=0}^N\Bbb V_j$ from \cref{lem4.1}. Then
$\Bbb V_0$ is a line in $\Bbb V$ transversal to $\Bbb V^1$ and hence defines a point
$\ell_0\in U$. We claim that the $P$-orbit of $\ell_0$ is all of $U$, while its
stabilizer subgroup in $P$ coincides with $G_0$. This shows that $U\cong P/G_0$ and
thus implies the first claimed description of $A\to M$. As observed in \cref{2.4}, an
element $g\in P$ can be written uniquely as $\exp(Z)g_0$ for $g_0\in G_0$ and
$Z\in\frak g_1$. We know that $\Bbb V_0$ is $G_0$-invariant, so for $w\in\ell_0$ we
get $g_0\cdot w=aw$ for some nonzero element $a\in\Bbb R$. On the other hand,
$\exp(Z)\cdot w=w+Z\cdot w+\tfrac12Z\cdot Z\cdot w+\dots$, and all but the first two
summands lie in $\Bbb V^2$. This shows that the action of $\exp(Z)g_0$ sends $\ell_0$
to the line in $\Bbb V/\Bbb V^2$ spanned by $w+Z\cdot w+\Bbb V^2$. But from
\cref{lem4.1}, we know that the action defines a surjection $\frak g_1\otimes\Bbb
V_0\to\Bbb V_1$, which shows that $P\cdot\ell_0=U$. On the other hand, it is well
known that $\frak g_1$ is an irreducible representation of $\frak g_0$. Thus also
$\frak g_1\otimes\Bbb V_0$ is irreducible, so $Z\cdot w=0$ if and only $Z=0$, which
shows that the stabilizer of $\ell_0$ in $P$ coincides with $G_0$.

For the second description, we need some input from the machinery of BGG
sequences. There is a natural invariant differential operator $S:\Ga(\Cal
EM)\to\Ga(\Cal VM)$ which splits the tensorial map $\Ga(\Cal VM)\to\Ga(\Cal EM)$
induced by the quotient projection $\Cal VM\to\Cal EM$. It turns out that the operator $\Gamma(\mathcal EM)\to\Gamma(\mathcal VM/\mathcal V^{i+1}M)$ induced by $S$ has
order $i$, so there is an induced vector bundle map from the jet prolongation
$J^i\Cal EM$ to $\Cal VM/\Cal V^{i+1}M$. As proved in \cite{BCEG}, the representation
$\Bbb V$ determines an integer $i_0$ such that this is an isomorphism for all $i\leq
i_0$. For a non-trivial representation $i_0\geq 1$, so we conclude that
$J^1\Cal EM\cong \Cal VM/\Cal V^2M$. Since $S$ splits the tensorial projection, we
see that, in a point $x\in M$, the hyperplane $\Cal V^1_xM/\Cal V^2_xM$ corresponds
to the jets of sections vanishing in $x$. Thus lines in $\Cal VM/\Cal V^2M$ that are
transversal to $\Cal V^1M/\Cal V^2M$ exactly correspond to lines in $J^1\Cal EM$ that
are transversal to the kernel of the natural projection to $\Cal EM$. Choosing such a
line is equivalent to choosing a splitting of this projection and thus of the jet
exact sequence for $J^1\Cal EM$. It is well known that the choice of such a splitting
is equivalent to the choice of a linear connection on $\Cal EM$.
\end{proof}

\begin{example}\label{ex4.1}
(1) \textbf{Oriented projective structures}. For $n\geq 2$, put $G:=SL(n+1,\Bbb R)$
  and let $P$ be the stabilizer of the ray in $\Bbb R^{n+1}$ spanned by the first
  element $e_0$ in the standard basis. Taking the complementary hyperplane spanned by
  the remaining basis vectors, one obtains a $|1|$-grading on the Lie algebra $\frak
  g$ of $G$ by decomposing into blocks of sizes $1$ and $n$ as in
  $\left(\begin{smallmatrix}\frak g_0 & \frak g_1\\ \frak g_{-1} & \frak g_0
  \end{smallmatrix}\right)$. The subgroup $G_0\subset P$ is then easily seen to
  consist of all block diagonal matrices in $P$.

  Now we define $\Bbb V:=\Bbb R^{(n+1)*}$, the dual of the standard representation of
  $G$. In terms of the dual of the standard basis, this decomposes as the sum of
  $\Bbb V_0:=\Bbb R\cdot e_0^*$ and $\Bbb V_1$ spanned by the remaining basis
  vectors. All properties claimed in \cref{lem4.1} are obviously satisfied in this
  case. The tractor bundle $\Cal VM$ corresponding to $\Bbb V$ is usually called the
  (standard) cotractor bundle $\Cal T^*M$ and the line bundle $\Cal EM$ is the bundle
  $\Cal E(1)$ of projective $1$-densities. Since $\Bbb V^2=\{0\}$ in this case,
  \cref{thm4.1} realizes $A$ as an open subbundle in $\Cal P(\Cal
  T^*M)$, and this is exactly the construction from \cite{Dunajski-Mettler}. In fact,
  it is well known that $\Cal T^*M\cong J^1\Cal E(1)$ in this case, this is even used
  as a definition in \cite{BEG}.

  More generally, for $k\geq 2$, we can take $\Bbb V$ to be the symmetric power
  $S^k\Bbb R^{(n+1)*}$. This visibly decomposes as $\Bbb V_0\oplus\dots\oplus\Bbb
  V_k$, where $\Bbb V_i$ is spanned by the symmetric products of $(e_0^*)^{k-i}$ with
  $i$ other basis elements. This corresponds to the tractor bundle $S^k\Cal T^*M$,
  while the quotient $\Bbb V/\Bbb V^1$ induces the $k$th power of $\Cal E(1)$, which
  is usually denoted by $\Cal E(k)$ and called the bundle of projective
  $k$-densities. Again, all properties claimed in \cref{lem4.1} are obviously
  satisfied.

  \medskip

  (2) \textbf{Conformal structures}. For $p+q=n\geq 3$, we put $G:=SO(p+1,q+1)$ and
  we take a basis $e_0,\dots,e_{n+1}$ for the standard representation $\Bbb R^{n+2}$
  of $G$ such that the non-trivial inner products are $\langle e_0,e_{n+1}\rangle=1$,
  and $\langle e_i,e_i\rangle=1$ for $1\leq i\leq p$ and $\langle e_i,e_i\rangle=-1$
  for $p+1\leq i\leq n$. Splitting matrices into blocks of sizes $1$, $n$, and $1$
  defines a $|1|$-grading of $\frak g$ according to $\left(\begin{smallmatrix} \frak
    g_0 & \frak g_1 & 0\\ \frak g_{-1} & \frak g_0 & * \\ 0 & * &
    * \end{smallmatrix}\right)$, where entries marked by $*$ are determined by other
  entries of the matrix. Again $G_0$ turns out to consist of block diagonal matrices
  and is isomorphic to the conformal group $CO(p+1,q+1)$ via the adjoint action on
  $\frak g_{-1}$. In view of \cref{rem3.1} we conclude that a parabolic
  geometry of type $(G,P)$ on a manifold $M$ is equivalent to a conformal
  structure. 

  The standard representation $\Bbb V$ of $G$ now decomposes as $\Bbb V=\Bbb
  V_0\oplus\Bbb V_1\oplus\Bbb V_2$, with the subspaces spanned by $e_0$,
  $\{e_1,\dots,e_n\}$, and $e_{n+1}$, respectively. This induces the \textit{standard
    tractor bundle} $\Cal TM$ on conformal manifolds, and the bundle induced by $\Bbb
  V/\Bbb V^1$ is the bundle $\Cal E[1]$ of \textit{conformal $1$-densities}. Hence
  \cref{thm4.1} in this case realizes $A$ as an open subbundle of the
  projectivization of the quotient of $\Cal TM$ by its smallest filtration
  component. Again, it is well known that this quotient is isomorphic to $J^1\Cal
  E[1]$.

  Alternatively, for $k\geq 2$, we can take $\Bbb V$ to be the trace-free part
  $S^k_0\Bbb R^{n+2}$ in the symmetric power of the standard representation. This
  leads to the bundle $S^k_0\Cal TM$ and the line bundle $\Cal E[k]$ of conformal
  densities of weight $k$. Here the decomposition from \cref{lem4.1} becomes a
  bit more complicated, since the individual pieces $\Bbb V_j$ are not irreducible
  representations of $G_0$ in general. Still the properties claimed in
  \cref{lem4.1} are obvious via the construction from the decomposition of the
  standard representation.

  \medskip

  (3) \textbf{Almost Grassmannian structures}. Here we choose integers $2\leq p\leq
  q$, put $n=p+q$, take  $G:=SL(n,\Bbb R)$ and $P\subset G$ the stabilizer of the
  subspace spanned by the first $p$ vectors of the standard basis of the standard
  representation $\Bbb R^n$ of $G$. Fixing the complementary subspace spanned by the
  remaining $q$ vectors in that basis, one obtains a decomposition of the Lie algebra
  $\frak g$ of $G$ into blocks of sizes $p$ and $q$, which defines a $|1|$-grading as
  in the projective case. Then $G_0$ again turns out to consist of block diagonal
  matrices and hence is isomorphic to $S(GL(p,\Bbb R)\x GL(q,\Bbb R))$, while $\frak
  g_{-1}$ can be identified with the space of $q\x p$-matrices endowed with the
  action of $G_0$ defined by matrix multiplication from both sides.

  Hence the corresponding geometries exist in dimension $pq$ and they are essentially
  given by an identification of the tangent bundle with a tensor product of two
  auxiliary bundles of rank $p$ and $q$, respectively, see Section 4.1.3 of
  \cite{book}. There it is also shown that for these types of structures the Cartan
  curvature has two fundamental components, but their nature depends on $p$ and
  $q$. For $p=q=2$, such a structure is equivalent to a split-signature conformal
  structure, so we will not discuss this case here. If $p=2$ and $q>2$, then one of
  these quantities is the intrinsic torsion of the structure, but the second is a
  curvature, so this is a case in which there are non-flat, torsion-free
  examples. For $p>3$, the intrinsic torsion splits into two components, and
  torsion-freeness of a geometry implies local flatness.

  The basic choice of a representation $\Bbb V$ that \cref{thm4.1} can be
  applied to is given by $\La^p\Bbb R^{n*}$, the $p$th exterior power of the dual of
  the standard representation. This decomposes as $\Bbb V=\Bbb
  V_0\oplus\dots\oplus\Bbb V_p$, where $\Bbb V_j$ is spanned by wedge products of
  elements of the dual of the standard basis that contain $p-j$ factors from
  $\{e_1^*,\dots,e_p^*\}$. The properties claimed in \cref{lem4.1} can be easily
  deduced from the construction from the dual of the standard representation.
\end{example}

\subsection{The projective Monge-Amp\`ere equation}\label{4.2}
This is the prototypical example of the non-linear PDE that we want to study. In the
setting of projective geometry, we have met the density bundles $\Cal E(k)$ for $k>0$
in \cref{ex4.1}. We define $\Cal E(-k)$ to be the line bundle dual to $\Cal E(k)$ and
use the convention that adding ``$(k)$'' to the name of a bundle indicates a tensor
product with $\Cal E(k)$ for $k\in\Bbb Z$. The first step towards the construction of
the projective Monge-Amp\`ere equation is that there is a projectively invariant,
linear, second order differential operator $H:\Ga(\Cal E(1))\to \Ga(S^2T^*M(1))$
called the \textit{projective Hessian}. Indeed, this is the first operator in the BGG
sequence determined by the standard cotractor bundle, see \cite{Proj-Comp}.

Now for a section $\si\in\Ga(\Cal E(1))$, $H(\si)$ defines a symmetric bilinear form
on each tangent space of $M$, and such a form has a well defined determinant. In projective geometry, this determinant admits an interpretation as a
density as follows. In the setting of part (1) of \cref{ex4.1}, the top
exterior power $\La^{n+1}\Bbb R^{(n+1)*}$ is a trivial representation, which implies
that the bundle $\La^{n+1}\Cal T^*M$ is canonically trivial. Identifying $\Cal T^*M$
with $J^1\Cal E(1)$, the jet exact sequence $0\to T^*M(1)\to J^1\Cal E(1)\to\Cal
E(1)\to 0$ implies that $\La^{n+1}\Cal T^*M\cong\La^nT^*M(n+1)$, so
$\La^nT^*M\cong\Cal E(-n-1)$. This isomorphism can be encoded as a tautological
section of $\La^nTM(-n-1)$. To form the determinant of $H(\si)$, one now takes the
tensor product of two copies of this canonical section and of $n$ copies of $H(\si)$
and forms the unique (potentially) non-trivial complete contraction of the result (so
the two indices of each copy of $H(\si)$ have to be contracted into different copies
of the tautological form). This shows that $\det(H(\si))$ can be naturally
interpreted as a section of $\Cal E(-n-2)$.

Assuming that $\si\in\Ga(\Cal E(1))$ is nowhere vanishing, we can form
$\si^k\in\Ga(\Cal E(k))$ for any $k\in\Bbb Z$, and hence 
\begin{equation}\label{eq:projmongeampere}
\det(H(\si))=\pm\si^{-n-2}
\end{equation} 
is a projectively invariant, fully non-linear PDE on nowhere vanishing sections of
$\Cal E(1)$. Observe that multiplying $\si$ by a constant, the two sides of the
equation scale by different powers of the constant, so allowing a constant factor instead of just a sign in the right hand side of the equation would only be a trivial
modification.

\subsection{Interpretation in terms of Weyl structures}\label{4.3}
Let us first observe that a nowhere vanishing section $\si\in\Cal E(1)$ uniquely
determines a Weyl structure. In the language of \cref{thm4.1} this can be
either described as the structure corresponding to the flat connection on $\Cal E(1)$
determined by $\si$ or as the one corresponding to the line in $\Cal T^*M$ spanned by
$S(\si)$, where $S$ denotes the BGG splitting operator. From either interpretation it
is clear that this Weyl structure remains unchanged if $\si$ is multiplied by a
non-zero constant. Alternatively, one can easily verify that any projective class on
$M$ contains a unique connection such that $\si$ is parallel for the induced
connection on $\Cal E(1)$.

To deal with non-flat cases in the following theorem, we use a concept of mean
curvature tailored to the case of connections compatible with an almost bi-Lagrangian
structure that was introduced in \cite{Smoczyk}. That article uses the terminology of
(almost) para-K\"ahler structures, which is slightly different from ours, but it is
easy to translate between the two.

\begin{thm}\label{thm4.3}
Let $M$ be an oriented smooth manifold of dimension $n$ which is endowed with a
projective structure. Let $\si\in\Ga(\Cal E(1))$ be a nowhere-vanishing section and
let us denote by $\nabla^\si$ the Weyl connections of the Weyl structure determined
by $\si$ and by $\Rho^{\si}$ its Rho tensor. Then we have:

(1) An appropriate constant multiple of $\si$ satisfies \eqref{eq:projmongeampere} if
and only if $\nabla^\si(\det(\Rho^{\si}))=0$ and $\det(\Rho^\si)$ is nowhere
vanishing.

(2) The Weyl structure determined by $\si$ is always Lagrangian. If $M$ is projectively
flat, then an appropriate constant multiple of $\si$ satisfies
\eqref{eq:projmongeampere} if and only if this Weyl structure is non-degenerate and
the image of the corresponding section $s:M\to A$ is a minimal submanifold.
 This extends to curved projective structures provided that minimality of $s(M)$ is
 defined as vanishing of the mean curvature form associated to the canonical
 connection $D$ via the definition in \cite[p.\ 120]{Smoczyk}.
\end{thm}
\begin{proof}
  It is well known how to decompose the curvature of a linear connection on $TM$ into
  the projective Weyl curvature and the Rho-tensor, see Section 3.1 of \cite{BEG}
  (taking into account the sign conventions mentioned in \cref{2.6}). It is also
  shown there that the Bianchi identity shows that the skew part of the Rho tensor is
  a non-zero multiple of the trace of the curvature tensor, which describes the
  action of the curvature on the top exterior power of the tangent bundle and thus on
  density bundles. This readily implies that $\Rho^\si$ is symmetric, so the Weyl
  structure defined by $\si$ is Lagrangian by \cref{prop3.4}.

  (1) It is well known that the projective Hessian in terms of a linear connection
  $\nabla$ in the projective class and its Rho tensor $\Rho$ is given by the
  symmetrization of $\nabla^2\si-\Rho\si$, see Section 3.2 of \cite{Proj-Comp}. (The
  different sign is caused by different sign conventions for the Rho-tensor.) But by
  definition $\nabla^\si\si=0$ and $\Rho^{\si}$ is symmetric, so we conclude that
  $H(\si)=-\Rho^\si\si$ and hence $\det(H(\si))=(-1)^n\si^n\det(\Rho^{\si})$. Now for
  each $k\neq 0$, the sections of $\Cal E(k)$ that are parallel for $\nabla^\si$ are
  exactly the constant multiples of $\si^k$, so (1) follows readily.

  (2) Since $\Rho^\si$ is symmetric, $\det(\Rho^\si)$ is nowhere vanishing if and
  only if $\Rho^\si$ is non-degenerate. By \cref{prop3.4}, this is
  equivalent to non-degeneracy of the Weyl structure determined by $\si$, and we
  assume this from now on. Let us write $\det(\Rho^{\si})$ in terms of the tautological
  section $\ep$ of $\La^nTM(-n-1)$ as above as
  $C(\ep\otimes\ep\otimes(\Rho^{\si})^{\otimes^n})$, where $C$ denotes the appropriate
  contraction. Applying $\nabla^\si$ to this, we observe that $C$ and $\ep$ are
  projectively invariant bundle maps and thus parallel for any Weyl connection. Thus
  we conclude that
  \[
  \nabla^\si\det(\Rho^{\si})=(n-1)C\left(\ep\otimes\ep\otimes(\Rho^{\si})^{\otimes^{n-1}}
  \otimes\nabla^\si\Rho^\si\right).
  \] Here we have used that the contraction is symmetric in the bilinear forms we
  enter and in the right hand side the form index of $\nabla^{\sigma}$ remains
  uncontracted. Since we assume that $\Rho^\si$ is invertible, linear algebra tells
  us that contracting two copies of $\ep$ with $(\Rho^{\si})^{\otimes^{n-1}}$ gives
  $\det(\Rho^\sigma)\Phi$, where $\Ph\in\Ga(S^2TM)$ is the inverse of
  $\Rho^\si$. Returning to the abstract index notation used in \cref{thm3.4.2} and
  writing $\nabla$ and $\Rho$ instead of $\nabla^\sigma$ and $\Rho^\sigma$, we
  conclude that $\nabla\det(\Rho)=0$ is equivalent to $0=\Rho^{ab}\nabla_i\Rho_{ab}$.

  On the other hand, we know the second fundamental form of $s(M)\subset A$ from
  \cref{thm3.4.2}. To determine the mean curvature, we have to contract an
  inverse metric into this expression, and we already know that this inverse metric
  is just $\Rho^{ij}$. Vanishing or non-vanishing of the result is
  independent of the final contraction with $\Rho^{kc}$. Thus we conclude that
  $s(M)\subset A$ is minimal if and only if
  \begin{equation}\label{eq:minlag}
  0=\Rho^{ab}(2\nabla_a\Rho_{bi}-\nabla_i\Rho_{ab}).
  \end{equation}
  In the proof of \cref{thm3.4.2}, we have also noted that the Cotton-York tensor is
  given by the alternation of $\nabla_i\Rho_{jk}$ in the first two indices. It is
  well known that this vanishes for projectively flat structures (see \cite{BEG}) and
  hence in the projectively flat case, $\nabla_i\Rho_{jk}$ is completely
  symmetric. Using this, the  
  claim in the projectively flat case  follows immediately.

  In the non-flat case, we first have to determine the map $\phi$ from Lemma 4 of
  \cite{Smoczyk}. In our notation, the map $P$ used there is given by
  $P|_{L^\pm}=\pm\id$. Using this, the beginning of the proof of \cref{thm3.4.1}
  readily shows that, in the notation used there, for $\xi\in T_xM\cong L^-_{s(x)}$,
  we get $\phi(\xi^i)=(\xi^i,-\Rho_{ja}\xi^a)$. Now we can combine this with the
  formula for $II_D$ from \cref{thm3.4.2}, which describes the operator
  $\widehat{A}$ used in \cite{Smoczyk}. This easily shows that, up to a non-zero
  factor, that the operator $\widehat{h}$ from \cite{Smoczyk} is given by
    $$
    \widehat{h}(\xi^i,\eta^j,\ze^k)=\xi^i\eta^j\ze^k\Rho_{ia}\Rho^{ab}\nabla^s_j\Rho_{kb}. 
    $$
    By definition, the mean curvature form $\widehat{H}$ from \cite{Smoczyk} is the
    trace over the first and third entry of this. Thus we have to contract
    $\widehat{h}_{ijk}$ with $\Rho^{ik}$, which again leads to
    $\Rho^{ab}\nabla^s_j\Rho_{ab}$.
\end{proof}
\begin{remark}
In the special case of a two-dimensional projective structure the minimality
condition \eqref{eq:minlag} was previously obtained in \cite[Theorem 4.4]{Mettler:MinLag}.
\end{remark}

A deep relation between solutions of the projective Monge-Amp\`ere equation and
properly convex projective structures was established in the works 
\cite{MR2402597} by Labourie and \cite{MR1828223} by Loftin. Recall that a projective manifold $(M,[\nabla])$ is
called~\textit{properly convex} if it arises as a quotient of a properly convex open
set $\tilde{M}\subset \mathbb{RP}^n$ by a group $\Gamma$ of projective
transformations which acts discretely and properly discontinuously.  The projective
line segments contained in $\tilde{M}$ project to $M$ to become the geodesics of
$[\nabla]$. Therefore, locally, the geodesics of a properly convex projective
structure $[\nabla]$ can be mapped diffeomorphically to segments of straight lines,
that is, $[\nabla]$ is locally flat. Combining the work of Labourie and Loftin  with \cref{thm4.3}, we obtain:

\begin{cor}\label{cor:propconvex}
Let $(M,[\nabla])$ be a closed oriented locally flat projective manifold. Then $[\nabla]$ is
properly convex if and only if $[\nabla]$ arises from a minimal Lagrangian Weyl
structure whose Rho tensor is positive definite.
\end{cor}
\begin{proof}
  Suppose that the flat projective structure $[\nabla]$ arises from a minimal Lagrangian
  Weyl structure $s$ whose Rho tensor $\Rho^s$ is positive definite. Since $s$ is
  Lagrangian, $\Rho^s$ is a constant negative multiple of the Ricci tensor of the
  Weyl connection $\nabla^s$, see \cite{BEG}, and taking into account the sign issue
  mentioned in \cref{2.6}.  By \cref{thm3.4.2}, the nowhere vanishing density
  $\det(\Rho^s)$ is preserved by $\nabla^s$, whence an appropriate power defines a
  volume density that is parallel for $\nabla^s$. Finally, projective flatness
  implies that the pair $(\nabla^s,\Rho^s)$ satisfies the hypothesis of Theorem 3.2.1
  of \cite{MR2402597}, which then implies that $[\nabla]$ is properly convex.

Conversely, suppose that $[\nabla]$ is properly convex. By \cite[Theorem 4]{MR1828223}
there is a solution $\sigma$ to the projective Monge-Amp\`ere equation with right hand
side $(-1)^{n+2}\si^{n+2}$ and such that $\si$ is negative for the natural
orientation on $\Cal E(1)$. By \cref{thm4.3}, $[\nabla]$ arises from a minimal
Lagrangian Weyl structure. Since the Hessian of $\sigma$ is positive definite, so is
the Rho tensor.
\end{proof}

\begin{remark}
Existence and uniqueness of minimal Lagrangian Weyl structures for a given torsion-free AHS structure is an interesting fully non-linear PDE problem. In the special case of projective surfaces, some partial results regarding uniqueness have been obtained in~\cite{MR3384876} and \cite{arXiv:1804.04616}. See also \cite{MR3968880} for a connection to dynamical systems and \cite{arXiv:1510.01043} for a related variational problem on the space of conformal structures. 
\end{remark}

\subsection{Invariant non-linear PDE for other AHS structures}\label{4.4}
We conclude this article with some remarks on analogs of the projective
Monge-Amp\`ere equation for other AHS structures. The first observation is that a
small representation theoretic condition is sufficient to obtain an analog of the
projectively invariant Hessian, which again is closely related to the Rho tensor.

To formulate this, we need a bit of background. Suppose that $(G,P)$ corresponds to a
$|1|$-grading of $G$ and let $G_0\subset P$ be the subgroup determined by the
grading. Then this naturally acts on each $\frak g_i$, and there is an induced
representation on $S^2\frak g_1$. We can decompose this representation into
irreducibles and there is a unique component whose highest weight is twice the
highest weight of $\frak g_1$. This is called the \textit{Cartan square} of $\frak
g_1$ and denoted by $\ocirc^2\frak g_1$. It comes with a canonical $G_0$-equivariant
projection $\pi:\otimes^2\frak g_1\to\ocirc^2\frak g_1$.  For any parabolic geometry
of type $(G,P)$, this induces a natural subbundle $\ocirc^2T^*M\subset S^2T^*M$ and a
natural bundle map $\pi:\otimes^2T^*M\to\ocirc^2T^*M$.

\begin{prop}\label{prop4.4}
Suppose that $(G,P)$ corresponds to a $|1|$-grading on the simple Lie algebra $\frak
g$ of $G$. Suppose further, that there is a representation $\Bbb V$ of $G$ satisfying
the assumptions of \cref{thm4.1} whose complexification is a fundamental
representation of the complexification of $\frak g$, and let $\Cal EM$ denote the
natural line bundle induced by $\Bbb V/\Bbb V^1$.

(1) There is an invariant differential operator $H:\Ga(\Cal
EM)\to\Ga(\ocirc^2T^*M\otimes\Cal EM)$ of second order.

(2) For a nowhere vanishing section $\si\in\Ga(\Cal EM)$, let $\Rho^\si$ be the Rho
tensor of the Weyl structure determined by $\si$. Then $H(\si)$ is a non-zero
multiple of $\pi(\Rho^\si)\si$, where $\pi$ is the projection to the Cartan square.
\end{prop}
\begin{proof}
(1) The representation $\Bbb V$ induces a tractor bundle on parabolic geometries of
  type $(G,P)$ to which the construction of BGG sequences can be applied. The first
  operator $H$ in the resulting sequence is defined on $\Gamma(\mathcal EM)$. Now the
  condition that $\dim(\Bbb V/\Bbb V^1)=1$ implies that the complexification of $\Bbb
  V$ is the fundamental representation corresponding to the simple root that induces
  the $|1|$-grading that defines $\frak p$. Since the complexification of $\Bbb V$ is
  a fundamental representation, the results of \cite{BCEG} show that the first
  operator in the BGG sequence has order two and the target space claimed in (1).

\smallskip

(2) It is also known in general (see \cite{AHS1} or \cite{CDS}) how to write out $H$
in terms of a Weyl structure with Weyl connection $\nabla^s$ and Rho tensor $\Rho^s$:
For $\si\in\Gamma(\mathcal EM)$, one then has to form $\nabla^2\si-\Rho^s\si$, symmetrize and
then project to the Cartan square. But if $s$ is the Weyl structure determined by
$\si$, then by definition $\nabla^s\si=0$ and $\Rho^s=\Rho^\si$, which implies the
claim.
\end{proof}
 
Observe that \cref{ex4.1} provides representations $\Bbb V$ that satisfy the
assumptions of the proposition for conformal and for almost Grassmannian
structures. Hence for these two geometries an invariant Hessian is available. It is
worth mentioning that, for conformal structures, $\pi(\Rho^\si)$ is the trace-free
part of $\Rho^\si$.

\medskip

It is also a general fact that the top-exterior power of $T^*M$ is isomorphic
to a positive, integral power of the dual $\Cal E^*M$ of $\Cal EM$: By definition, the
grading element $E$ acts by multiplication by $\dim(\frak g_1)$ on the top exterior
power of $\frak g_1$, which represents the top exterior power of $T^*M$. On the other
hand, the construction implies that a generator of $\Bbb V_0\subset\Bbb V$ will be a
lowest weight vector of the complexification of $\Bbb V$, so $E$ acts by a negative
number on this. The fact that we deal with a fundamental representation implies
that $\dim(\mathfrak g_{-1})$ is an integral multiple of that number. As in the projective case, this can be phrased as the
existence of a tautological section, which can then be used together with copies of
$H(\si)$ to obtain a section of a line bundle, which can be trivial, a tensor power
of $\Cal E$ or a tensor power of $\Cal E^*$. In any case, a nowhere vanishing section
of $\Cal E$ determines a canonical section of that bundle (which is the constant $1$
in the trivial case), so there is an invariant version of the Monge-Amp\`ere
equation. In view of part (2) of \cref{prop4.4}, for these equations there
is always an analog of part (1) of \cref{thm4.3}. 

For some of the structures, there are additional natural sections that can be used
together with powers of $H(\si)$ to construct other non-linear invariant operators,
for example, the conformal metric for conformal structures and partial (density
valued) volume forms for Grassmannian structures. Again  part (2) of
\cref{prop4.4} shows that all these equations can be phrased as equations on
$\Rho^\si$, so there should be a relation to submanifold geometry of Weyl structures
in the style of part (2) of \cref{thm4.3}. All this will be taken up in detail
elsewhere.


\begin{bibdiv}
  \begin{biblist}

  \bib{MR2545240}{article}{
   author={Bader, Uri},
   author={Frances, Charles},
   author={Melnick, Karin},
   title={An embedding theorem for automorphism groups of Cartan geometries},
   journal={Geom. Funct. Anal.},
   volume={19},
   date={2009},
   number={2},
   pages={333--355},
   issn={1016-443X},
   review={\MR{2545240}},
}

  \bib{BEG}{article}{
   author={Bailey, T. N.},
   author={Eastwood, M. G.},
   author={Gover, A. R.},
   title={Thomas's structure bundle for conformal, projective and related
   structures},
   journal={Rocky Mountain J. Math.},
   volume={24},
   date={1994},
   number={4},
   pages={1191--1217},
   issn={0035-7596},
   review={\MR{1322223}},
}   

\bib{BastonI}{article}{
   author={Baston, R. J.},
   title={Almost Hermitian symmetric manifolds. I. Local twistor theory},
   journal={Duke Math. J.},
   volume={63},
   date={1991},
   number={1},
   pages={81--112},
   issn={0012-7094},
   review={\MR{1106939}},
}

\bib{BastonII}{article}{
   author={Baston, R. J.},
   title={Almost Hermitian symmetric manifolds. II. Differential invariants},
   journal={Duke Math. J.},
   volume={63},
   date={1991},
   number={1},
   pages={113--138},
   issn={0012-7094},
   review={\MR{1106940}},
}

\bib{BCEG}{article}{
   author={Branson, Thomas},
   author={\v{C}ap, Andreas},
   author={Eastwood, Michael},
   author={Gover, A. Rod},
   title={Prolongations of geometric overdetermined systems},
   journal={Internat. J. Math.},
   volume={17},
   date={2006},
   number={6},
   pages={641--664},
   issn={0129-167X},
   review={\MR{2246885}},
}

\bib{MR1824987}{article}{
   author={Bryant, Robert L.},
   title={Bochner-K\"{a}hler metrics},
   journal={J. Amer. Math. Soc.},
   volume={14},
   date={2001},
   number={3},
   pages={623--715},
   issn={0894-0347},
   review={\MR{1824987}},
}

\bib{MR0365607}{article}{
   author={Calabi, Eugenio},
   title={Complete affine hyperspheres. I},
   conference={
      title={Symposia Mathematica, Vol. X},
      address={Convegno di Geometria Differenziale, INDAM, Rome},
      date={1971},
   },
   book={
      publisher={Academic Press, London},
   },
   date={1972},
   pages={19--38},
   review={\MR{0365607}},
}

\bib{MR3210600}{article}{
   author={Calderbank, David M. J.},
   title={Selfdual 4-manifolds, projective surfaces, and the Dunajski-West
   construction},
   journal={SIGMA Symmetry Integrability Geom. Methods Appl.},
   volume={10},
   date={2014},
   pages={Paper 035, 18},
   issn={1815-0659},
   review={\MR{3210600}},
}

  \bib{Calderbank-Diemer}{article}{
   author={Calderbank, David M. J.},
   author={Diemer, Tammo},
   title={Differential invariants and curved Bernstein-Gelfand-Gelfand sequences},
   journal={J. Reine Angew. Math.},
   volume={537},
   date={2001},
   pages={67--103},
   issn={0075-4102},
   review={\MR{1856258}},
}
    

 \bib{CDS}{article}{
   author={Calderbank, David M. J.},
   author={Diemer, Tammo},
   author={Sou\v{c}ek, Vladim\'{\i}r},
   title={Ricci-corrected derivatives and invariant differential operators},
   journal={Differential Geom. Appl.},
   volume={23},
   date={2005},
   number={2},
   pages={149--175},
   issn={0926-2245},
   review={\MR{2158042}},
 }

 \bib{Proj-Comp}{article}{
   author={\v{C}ap, Andreas},
   author={Gover, A. Rod},
   title={Projective compactifications and Einstein metrics},
   journal={J. Reine Angew. Math.},
   volume={717},
   date={2016},
   pages={47--75},
   issn={0075-4102},
   review={\MR{3530534}},
 }

 \bib{hol-red}{article}{
   author={\v{C}ap, A.},
   author={Gover, A. R.},
   author={Hammerl, M.},
   title={Holonomy reductions of Cartan geometries and curved orbit
   decompositions},
   journal={Duke Math. J.},
   volume={163},
   date={2014},
   number={5},
   pages={1035--1070},
   issn={0012-7094},
   review={\MR{3189437}},
}

 
\bib{Einstein}{article}{
   author={\v{C}ap, A.},
   author={Gover, A. R.},
   author={Macbeth, H. R.},
   title={Einstein metrics in projective geometry},
   journal={Geom. Dedicata},
   volume={168},
   date={2014},
   pages={235--244},
   issn={0046-5755},
   review={\MR{3158041}},
}



\bib{book}{book}{
   author={{\v{C}}ap, Andreas},
   author={Slov{\'a}k, Jan},
   title={Parabolic geometries. I},
   series={Mathematical Surveys and Monographs},
   volume={154},
   note={Background and general theory},
   publisher={American Mathematical Society},
   place={Providence, RI},
   date={2009},
   pages={x+628},
   isbn={978-0-8218-2681-2},
   review={\MR{2532439}},
}


\bib{Weyl}{article}{
   author={\v{C}ap, Andreas},
   author={Slov\'{a}k, Jan},
   title={Weyl structures for parabolic geometries},
   journal={Math. Scand.},
   volume={93},
   date={2003},
   number={1},
   pages={53--90},
   issn={0025-5521},
   review={\MR{1997873}},
}

\bib{AHS1}{article}{
   author={\v{C}ap, A.},
   author={Slov\'{a}k, J.},
   author={Sou\v{c}ek, V.},
   title={Invariant operators on manifolds with almost Hermitian symmetric
   structures. I. Invariant differentiation},
   journal={Acta Math. Univ. Comenian. (N.S.)},
   volume={66},
   date={1997},
   number={1},
   pages={33--69},
   issn={0862-9544},
   review={\MR{1474550}},
}


\bib{CSS-BGG}{article}{
   author={{\v{C}}ap, Andreas},
   author={Slov{\'a}k, Jan},
   author={Sou{\v{c}}ek, Vladim{\'{\i}}r},
   title={Bernstein-Gelfand-Gelfand sequences},
   journal={Ann. of Math.},
   volume={154},
   date={2001},
   number={1},
   pages={97--113},
   issn={0003-486X},
   review={\MR{1847589}},
}

\bib{Smoczyk}{article}{
    AUTHOR = {Chursin, Mykhaylo},
    AUTHOR = {Sch\"{a}fer, Lars},
    AUTHOR = {Smoczyk, Knut},
     TITLE = {Mean curvature flow of space-like {L}agrangian submanifolds in
              almost para-{K}\"{a}hler manifolds},
   JOURNAL = {Calc. Var. Partial Differential Equations},
    VOLUME = {41},
      YEAR = {2011},
    NUMBER = {1-2},
     PAGES = {111--125},
      ISSN = {0944-2669},
  review = {\MR{2782799}},
}

\bib{Dunajski-Mettler}{article}{
   author={Dunajski, Maciej},
   author={Mettler, Thomas},
   title={Gauge theory on projective surfaces and anti-self-dual Einstein
   metrics in dimension four},
   journal={J. Geom. Anal.},
   volume={28},
   date={2018},
   number={3},
   pages={2780--2811},
   issn={1050-6926},
   review={\MR{3833818}},
}

\bib{Frances}{article}{
   author={Frances, Charles},
   title={Sur le groupe d'automorphismes des g\'{e}om\'{e}tries paraboliques de rang
   1},
   journal={Ann. Sci. \'{E}cole Norm. Sup. (4)},
   volume={40},
   date={2007},
   number={5},
   pages={741--764},
   issn={0012-9593},
   review={\MR{2382860}},
}

\bib{Herzlich}{article}{
   author={Herzlich, Marc},
   title={Parabolic geodesics as parallel curves in parabolic geometries},
   journal={Internat. J. Math.},
   volume={24},
   date={2013},
   number={9},
   pages={1350067, 16},
   issn={0129-167X},
   review={\MR{3109439}},
}

\bib{MR2854277}{article}{
    AUTHOR = {Hildebrand, Roland},
     TITLE = {Half-dimensional immersions in para-{K}\"{a}hler manifolds},
   JOURNAL = {Int. Electron. J. Geom.},
    VOLUME = {4},
      YEAR = {2011},
    NUMBER = {2},
     PAGES = {85--113},
      ISSN = {1307-5624},
   review = {\MR{2854277}},
}
\bib{MR2854275}{article}{
    AUTHOR = {Hildebrand, Roland},
     TITLE = {The cross-ratio manifold: a model of centro-affine geometry},
   JOURNAL = {Int. Electron. J. Geom.},
  
    VOLUME = {4},
      YEAR = {2011},
    NUMBER = {2},
     PAGES = {32--62},
      ISSN = {1307-5624},
   review = {\MR{2854275}},
}

\bib{KN1}{article}{
   author={Kobayashi, Shoshichi},
   author={Nagano, Tadashi},
   title={On filtered Lie algebras and geometric structures. I},
   journal={J. Math. Mech.},
   volume={13},
   date={1964},
   pages={875--907},
   review={\MR{0168704}},
}

\bib{MR2402597}{article}{
   author={Labourie, Fran\c{c}ois},
   title={Flat projective structures on surfaces and cubic holomorphic
   differentials},
   journal={Pure Appl. Math. Q.},
   volume={3},
   date={2007},
   number={4, Special Issue: In honor of Grigory Margulis.},
   pages={1057--1099},
   issn={1558-8599},
   review={\MR{2402597}},
}

\bib{MR1828223}{article}{
   author={Loftin, John C.},
   title={Affine spheres and convex $\Bbb{RP}^n$-manifolds},
   journal={Amer. J. Math.},
   volume={123},
   date={2001},
   number={2},
   pages={255--274},
   issn={0002-9327},
   review={\MR{1828223}},
}

\bib{arXiv:1510.01043}{article}{
   author={Mettler, Thomas},
   title={Extremal conformal structures on projective surfaces},
   journal={Ann. Sc. Norm. Super. Pisa Cl. Sci. (5)},
   volume={20},
   date={2020},
   number={4},
   pages={1621--1663},
   issn={0391-173X},
   review={\MR{4201191}},
}

\bib{MR3384876}{article}{
   author={Mettler, Thomas},
   title={Geodesic rigidity of conformal connections on surfaces},
   journal={Math. Z.},
   volume={281},
   date={2015},
   number={1-2},
   pages={379--393},
   issn={0025-5874},
   review={\MR{3384876}},
}

\bib{Mettler:MinLag}{article}{
   author={Mettler, Thomas},
   title={Minimal Lagrangian connections on compact surfaces},
   journal={Adv. Math.},
   volume={354},
   date={2019},
   pages={106747, 36~pp.},
   issn={0001-8708},
   review={\MR{3987443}},
}

\bib{arXiv:1804.04616}{article}{
  Author = {Mettler, Thomas},
  author = {Paternain, Gabriel P.},
  Title = {Convex projective surfaces with compatible {W}eyl connection are hyperbolic},
  journal = {Anal.~PDE},
  number = {4},
  pages ={1073--1097},
  volume = {13},
  date = {2020},
  review={\MR{4109900}},
}

\bib{MR3968880}{article}{
   author={Mettler, Thomas},
   author={Paternain, Gabriel P.},
   title={Holomorphic differentials, thermostats and Anosov flows},
   journal={Math. Ann.},
   volume={373},
   date={2019},
   number={1-2},
   pages={553--580},
   issn={0025-5831},
   review={\MR{3968880}},
}

\bib{arXiv:1803.06870}{article}{
   author={Wienhard, Anna},
   title={An invitation to higher Teichm\"{u}ller theory},
   conference={
      title={Proceedings of the International Congress of
      Mathematicians---Rio de Janeiro 2018. Vol. II. Invited lectures},
   },
   book={
      publisher={World Sci. Publ., Hackensack, NJ},
   },
   date={2018},
   pages={1013--1039},
   review={\MR{3966798}},
}
\end{biblist}
\end{bibdiv}
\end{document}